\documentclass[11pt]{article}
\pdfoutput=1
\usepackage{amsmath,amsthm,amsfonts,amssymb,amsthm,bm,fixmath}
\usepackage{mathrsfs}
\usepackage{times}
\usepackage{enumitem}
\usepackage{esint}

\newtheorem{theorem}{Theorem}
\newtheorem{corollary}[theorem]{Corollary}

\newtheorem{lemma}[theorem]{Lemma}

\newtheorem*{theorem*}{Theorem}
\newtheorem*{lemma*}{Lemma}
\newtheorem{definition}[theorem]{Definition}
\newtheorem*{definition*}{Definition}
\theoremstyle{definition}

\newtheorem*{remark*}{Remark}
\newtheorem{remark}[theorem]{Remark}
\newtheorem*{remarks*}{Remarks}

\newcommand{\U}{\Upsilon}

\newcommand{\Om}{\Omega}
\newcommand{\om}{\omega}
\newcommand{\ol}{\overline}
\renewcommand{\forall}{\text{ for all }}

\def\NN{\mathbb N}

\newcommand{\wk}{\rightharpoonup}

\renewcommand{\leq}{\leqslant}
\renewcommand{\geq}{\geqslant}

\renewcommand{\mod}{~\text{mod}~}

\def\RR{\mathbb R}

\def\l{\left}

\newcommand{\e}{\epsilon}

\newcommand{\sF}{\mathcal F}

\newcommand{\sE}{\mathcal{E}}

\renewcommand{\a}{\alpha}

\renewcommand{\d}{\delta}

\newcommand{\half}{{\frac{1}{2}}}

\newcommand{\sR}{\mathcal{R}}

\renewcommand{\l}{\lambda}

\newcommand{\CC}{\mathbb{C}}

\newcommand{\sN}{\mathcal N}

\newcommand{\nc}{\newcommand}

\newcommand{\Lid}{L_\infty^*(X, \sL, \l)}
\newcommand{\Li}{L_\infty (X, \sL, \l)}
\newcommand{\Fk}{\mathfrak F}
\newcommand{\mG}{\mathfrak G}

\nc{\RLp}{\mbox{$L_{p}(\Gamma)\;$}}
\nc{\RLh}{\mbox{$L_{2}(\Gamma)\;$}} \nc{\sL}{\mathcal L}
\nc{\RLpr}{\mbox{$L_{p}(\Gamma,\rho)\;$}}
\nc{\RLhr}{\mbox{$L_{2}(\Gamma,\rho)\;$}} \nc{\ds}{\displaystyle}

\def\sin{\mathop{\rm sin}\nolimits}
\def\cos{\mathop{\rm cos}\nolimits}

\def\inf{\mathop{\rm inf}\nolimits}

\def\sup{\mathop{\rm sup}\nolimits}

\nc{\la}{\label} \nc{\sB}{\mathcal B}
\nc{\mS}{\mathscr S}

\numberwithin{equation}{section}
\numberwithin{theorem}{section}
\newcommand{\ed}{\end{document}}
\begin{document}

\title{Localizing  Weak Convergence in $\boldsymbol{ L_\infty}$}
\author{J. F. Toland\thanks{Department of Mathematical Sciences, University of Bath, Claverton Down, Bath, BA2 7AY, UK.\newline \tt{masjft@bath.ac.uk}}}
\maketitle 
\begin{abstract}
In  a general measure space $(X,\mathcal L,\lambda)$, a characterization of weakly null sequences  in $L_\infty (X,\mathcal L,\lambda)$ ($u_k \rightharpoonup  0$) in terms of their
pointwise behaviour almost everywhere is derived  from the Yosida-Hewitt identification of $L_\infty (X,\mathcal L,\lambda)^*$ with finitely additive measures,
and  extreme points of the unit ball in $L_\infty (X,\mathcal L,\lambda)^*$ with  $\pm \mathfrak G$,  where $\mathfrak G$ denotes the set of finitely additive measures that take only values 0 or $ 1$.
When  $(X,\tau)$ is  a locally compact Hausdorff  space with Borel $\sigma$-algebra $\mathcal B$, the well-known identification of $\mathfrak G$ with ultrafilters means that
this criterion for nullity is  equivalent to localized behaviour on   open neighbourhoods of   points $x_0$ in the one-point compactification of $X$. Notions of weak convergence at $x_0$ and the essential range of $u$ at
$x_0$ are natural consequences.
When  a  finitely additive measure $\nu$  represents
$f \in L_\infty(X, \mathcal B, \lambda)^*$    and $\hat \nu$  is the Borel measure representing  $f$ restricted  to $C_0(X,\tau)$, a minimax formula for $\hat \nu$ in terms $\nu$ is derived and
those  $\nu$ for which $\hat \nu$ is singular with respect to $\lambda$ are characterized.
\end{abstract}

\emph{Keywords:}
Dual of $L_\infty$; finitely additive measures; weak convergence; ultrafilters; pointwise \\convergence;  extreme points

\emph{AMS Subject Classification:}
46E30, 28C15, 46T99, 26A39, 28A25, 	46B04

\section{Introduction}

In  the usual Banach space $C(Z)$  of real-valued continuous functions on a compact metric  space $Z$ with the maximum norm, it is well-known \cite{banach} that $v_k$ converges weakly to $v$ ($v_k \wk v$)  if and only if  $\{\|v_k\|\}$ is bounded and $v_k(z) \to v(z)$ for all $z \in Z$.
This observation amounts to a simple test for weak convergence in $C(Z)$ from which follows, for example, the  weak sequential continuity \cite{ball} of  composition maps $u \mapsto f\circ u,~u \in C(Z)$, when $f: \RR \to \RR$ is continuous. However $u_k \wk u$ in $\Li$ implies  that $\{\|u_k\|_\infty\}$ is bounded and  often that $u_k(x) \to u(x)$ almost everywhere (Lemma \ref{ap1}),  but the converse is false (Remark \ref{app3}) and, despite the identification of $\Li$ with $C(Z)$ for some compact $Z$ \cite[VIII 2.1]{conway},  it can be difficult to decide whether or not a given sequence is weakly convergent in $\Li$.
  To address this issue Theorem \ref{thmIFF}  characterises sequences that are weakly convergent to 0 in $\Li$ (hereafter referred to as weakly null) purely in terms of their pointwise behaviour almost everywhere, and
a practical test for weak nullity ensues (Corollary \ref{cor3.6} and Section \ref{S3.1}).
  When   $(X,\tau)$ is a locally compact Hausdorff  topological space,
  localization in terms of opens sets, as opposed to pointwise, follows from the identification of ultrafilters in  the corresponding Borel measure space $(X,\sB,\l)$ with extreme points in the unit ball of $L_\infty(X,\sB,\l)^*$.
   When  $\nu$ is the finitely additive measure corresponding to  $f \in L_\infty(X,\sB,\l)^*$ we give a formula for the
   Borel measure $\hat\nu$  that   represent  the restriction $\hat f$ of $f$  to $C_0(X,\tau)$,  defined in Section \ref{restriction}, and use it to characterize those $\nu$ for which $\hat \nu$ is singular relative to $\l$.
   These observations are motivated by examples  \cite{ hensgen,valadier} of  singular finitely additive measures that do not yield singular Borel measures when restricted to continuous functions,
contrary to a claim by Yosida \& Hewitt \cite[Thm. 3.4]{yosidaetal}. The material is organized as follows.

Section \ref{Linfty*} is a brief survey of finitely additive measures on $\sigma$-algebras and of weak convergence in $\Li$ in terms of the Yosida-Hewitt  representation   of the dual space $\Li^*$ as a space $\Lid$ of finitely additive measures. When $\mG$ denotes elements of $\Lid$   that take only values $\{0,1\}$, it follows  that $u_k \wk u$ in $\Li$ if and only if $f(u_k) \to f(u)$ for all  $f$ represented by elements of $\mG$. Although obtained independently, this is a special case of Rainwater's Theorem, see Appendix and the Closing Remarks at the end of the paper.
The section ends with a brief account of weak sequential continuity of composition operators.

Section \ref{wkcgce} begins by remarking that $u_k \wk u$ if and only if $|u_k| \wk |u|$, noting aspects of the pointwise behaviour of weakly convergent sequences in $\Li$, and observing that a necessary condition, which turns out to be sufficient, is given by Mazur's theorem.
The characterization of null sequences in terms of their pointwise behaviour  in Theorem \ref{thmIFF} follows   from
 Yosida-Hewitt  theory and the fact that any $u \in \Li$ is a constant $\om$-almost everywhere in the sense of finitely additive measures when $\om \in \mG$ (see Remark following Theorem \ref{ans}).
An $\Li$ analogue of Dini's theorem on the uniform convergence of sequences of continuous functions that are monotonically convergent pointwise is a corollary, and Theorem \ref{thmIFF} is illustrated by several examples.

 In
Section \ref{topsection}, when $(X,\tau)$ is a locally compact Hausdorff space and  $(X, \sB,\l)$ is the corresponding Borel measure space, the well-known
 one-to-one correspondence  \eqref{11ultra}  between  $\mG$  and   a set $\Fk$ of ultrafilters (Definition \ref{ultrafilter})
leads to a
   local description of   weak convergence: a sequence  is  weakly convergent in $L_\infty(X,\sB,\l)$ if and only if it is weakly convergent at each  $x_0\in X_\infty$, the one-point compactification of $X$.
This notion of weak convergence at a point leads naturally to a definition
of the essential range $\sR(u)(x_0)$  of $u$ at $x_0 \in X_\infty$. For the relation between weak convergence and the pointwise essential range, see Remark \ref{necessuf}.

For $\nu \in L^*_\infty(X,\sB,\l)$,  let $\hat \nu$ denote the Borel measure that, by the Riesz Representation Theorem  \cite[Thm. 6.19]{rudin}), corresponds to the restriction to $C_0(X,\tau)$ of the  functional defined on $\nu$ on $L_\infty(X, \sB,\l)$ by \eqref{Le}. Section \ref{restriction} develops a  minimax formula (Theorem \ref{measures})  for  $\hat \nu$  in terms of $\nu$. It follows that if $(X,\tau)$ is not compact,  $\hat \nu$ may be zero when $\nu\geq 0$ is non-zero. In particular when $\om \in \mG$,  either $\hat \om = 0$ or $\hat\om \in \mathfrak D$ (a Dirac measure on $X$) and  if $(X, \tau)$ is compact $\hat \om \in \mathfrak D$. An arbitrary Hahn-Banach extension to $L_\infty(X,\sB, \l)$  of a $\d$-function on $C_0(X,\tau)$  need not be in $\mG$, but from Section \ref{localisation} there may be infinitely many extensions that are in $\mG$.
 Those $\nu$ for which $\hat\nu$ is singular with respect to $\l$ are characterised in Corollary \ref{singular1}.

\section{$\boldsymbol L_{\boldsymbol\infty}$ \bf and its Dual}\label{Linfty*}
Let $\l$ be a  non-negative, complete, countably additive measure  on  a $\sigma$-algebra $\sL$ in a set $X$ and let
$\sN = \{E \in \sL:\l(E) = 0\}$. So $(X,\sL, \l)$ is a measure space  and $\sN$ denotes its null sets.
As usual $(L_\infty(X, \sL, \l), \|\cdot\|_\infty)$ denotes  the corresponding Banach space of (equivalence classes of) essentially bounded functions. In  notation  summarised in Section \ref{fams}, the analogue of the Riesz Representation Theorem  \cite[Thm. 6.19]{rudin} for functionals in  $L_\infty(X,\sL,\l)^*$  is the following.

\begin{theorem}\textnormal{(Yosida \& Hewitt \cite{yosidaetal}, see also \cite[Theorem IV.8.16]{dunford})}. \label{yhthm} For every bounded linear functional on $L_\infty(X, \sL, \l)$ there exists a finitely additive measure (Definition \ref{llperp}) $\nu$ on $\sL$  such that
\begin{gather}\label{Le} f(w) = \int_X w \,d\nu \forall w \in L_\infty(X, \sL, \l),\\ \nu(N) = 0 ~\forall N \in \sN \text{ and } |\nu|(X) = \|f\|_\infty < \infty.\notag
\end{gather}
Conversely if $\nu$ is a finitely additive measure on $X$ with $\nu(N) =0$ for all $N \in \sN$, then $f$ defined by \eqref{Le} is in   $L_\infty(X, \sL, \l)^*$. We write $\nu \in \Lid$ if \eqref{Le} holds for some $f\in\Li^*$. \qed
\end{theorem}
Because $\nu$ is finitely additive, but not necessarily  $\sigma$-additive, integrals in \eqref{Le} should be treated with care. For example, the Monotone Convergence Theorem and Fatou's Lemma do not hold, and the Dominated Convergence Theorem holds only in a restricted form. The next section is a  review  of notation and standard theory; for a comprehensive account see \cite{yosidaetal}, \cite[Ch. III]{dunford} or \cite[Ch. 4]{raorao}.
When combined with the Hahn-Banach theorem, Theorem \ref{yhthm} yields the existence of a variety of finitely additive measures.
\subsection{Finitely Additive Measures: Notation and Definitions}\label{fams}
Although finitely additive measures are defined on algebras (closed under complementation and finite unions), here they are considered only on  $\sigma$-algebras, where their theory is somewhat more satisfactory, because $\sL$   in Theorem \ref{yhthm}  is a $\sigma$-algebra.
\begin{definition}\label{llperp}\textnormal{\cite[\S1.2-\S1.7]{yosidaetal}} A finitely additive measure  $\nu$ on $\sL$ is a mapping from $\sL$ into  $\RR$ with
\begin{gather*}\nu (\emptyset) =0 \text{ and }\sup_{A \in \sL} |\nu(A)| < \infty;\\ \nu(A\cup B) = \nu (A) + \nu (B) \forall A,\,B \in \sL\text{ with }A\cap B = \emptyset.
\end{gather*}
A finitely additive measure is  $\sigma$-additive if and only if
$$\nu\left(\cup_{k\in \NN}E_k\right) = \sum_{k\in \NN} \nu(E_k) \forall \{E_k\} \subset \sL \text{ with } E_j \cap E_k = \emptyset, j \neq k.  $$
Let $\Upsilon(\sL)$ and $\Sigma(\sL)$ denote, respectively, the families of finitely additive and $\sigma$-additive  measures on $\sL$.
\end{definition}
Since finitely-additive measures are not one-signed, the hypothesis that $\sup_{A \in \sL} |\nu(A)| < \infty$ does not follow from the fact that
$\nu(X) < \infty$.
  The following results are from \cite[\S 1.9-\S1.12]{yosidaetal}.

For $\nu_1,\,\nu_2 \in \Upsilon (\sL),\, E \in \sL$, let
\begin{subequations}\label{pmparts}
\begin{equation}
\begin{split}
(\nu_1\vee\nu_2)(E) &= \sup_{E \supset F \in \sL}\{\nu_1(F)+\nu_2(E\setminus F)\}, \\ (\nu_1\wedge\nu_2)(E)  &= -\big((-\nu_1)\vee (-\nu_2)\big)(E).
\end{split}
\end{equation}
Then $\nu_1\vee\nu_2,\, \nu_1\wedge\nu_2 \in \Upsilon(\sL)$, which is a lattice, and  $\nu\in \Upsilon(\sL)$ can be written
\begin{equation}\label{pm}\nu = \nu^+ -\nu^-\text{ where }  \nu^+ = \nu\vee 0,~~ \nu^- = (-\nu)\vee 0 \text{ and } \nu^+\wedge \nu^- = 0.\end{equation}
$\nu^\pm$ are the  positive and negative parts of $\nu$ and  $|\nu|: = \nu^++\nu_-$ is its total variation (see Theorem \ref{yhthm}).
\end{subequations}
 For $\nu_1,\,\nu_2 \in \Upsilon (\sL)$ write
$\nu_1 \ll \nu_2$ ($\nu_1$ is absolutely continuous with respect to $\nu_2$), if for all $\e>0$ there exists $\d$ such that $|\nu_1(E)| <\e$ when $|\nu_2|(E) < \d$, and write
$\nu_1 \perp \nu_2$ if for every $\e>0$ there exists $E \in \sL$ such that $|\nu_1|(E)+ |\nu_2|(X \setminus E) < \e$.

\begin{remark}\label{absol} When  $\nu_1,\nu_2 \in \Sigma(\sL) \subset\Upsilon(\sL)$ the above definitions imply:
\begin{align*}
&\text{$\nu_1 \ll \nu_2$ if and only if  $|\nu_2|(E) =0 $ implies $\nu_1(E) = 0$} \forall E \in \sL; \\&\text{$\nu_1\perp \nu_2$ ~if and only if $|\nu_1|(E)+ |\nu_2|(X \setminus E) = 0$ for some $E \in \sL$.}\end{align*}
However it is important that a non-negative finitely additive measure $\nu$ which vanishes on $\sN$ (see Theorems \ref{yhthm} and  \ref{thm1*}) need  not satisfy  $\nu \ll \l$ if $\nu \not\in \Sigma (\sL)$.
\qed\end{remark}
\begin{definition}\label{pfa}\textnormal{\cite[\S1.13]{yosidaetal}} A non-negative  $\nu \in \U(\sL)$ is  purely finitely additive (written $\nu \in \Pi(\sL)$) if
$$\left\{\gamma \in \Sigma(\sL): 0\leq\gamma \leq \nu\right\} = \{0\}.
$$
Equivalently,  $0\leq \nu \in\Pi(\sL)$ if and only if $\nu \wedge \gamma =0$ for all  $0\leq \gamma \in \Sigma (\sL)$. In general,
$\nu \in \U(\sL)$ is  purely finitely additive if  $\nu^+$ and $\nu^-$ are purely finitely additive.
\end{definition}
Note that $\Pi(\sL)\cap \Sigma(\sL) = \{0\}$ and if
 $\a \in \RR$ and $\nu \in \Pi(\sL)$ then $\a \nu \in \Pi(\sL)$.
Moreover $\Pi(\sL)$ is a lattice \cite[Thm. 1.17]{yosidaetal}: if $\nu_i \in \Pi(\sL)$, $i = 1,\,2$, then $\nu_1+\nu_2,\,\nu_1\wedge \nu_2,\,\nu_1\vee \nu_2 \in\Pi(\sL)$.
The sense in which a purely finitely additive measure \emph{on a $\sigma$-algebra} is singular with respect to any $\sigma$-additive measure is captured by the following observation which is not true if $\sL$ is only an algebra.
\begin{theorem}  \textnormal{\cite[Thm 1.22]{yosidaetal}} \label{1.22} For $0\leq \gamma \in \Sigma(\sL)$ and  $0\leq \mu \in \Pi(\sL)$ there exists $\{E_k\} \subset \sL$ with
$$E_{k+1} \subset E_k,\quad \mu(E_k) = \mu(X) \forall k \text{ and } \gamma (E_k) \to 0\text{ as }
k \to \infty.$$ Conversely if $0\leq \mu \in \Upsilon(\sL)$  and for all $0\leq \gamma \in \Sigma(\sL)$ a sequence $\{E_k\}$ with these properties exists, then $\mu \in \Pi(\sL)$.
\end{theorem}
The significance of purely finitely additive measures is evident from the following.
\begin{theorem} \textnormal{\cite[Thms 1.23 \& 1.24]{yosidaetal}} \label{decomposition} Any $\nu \in \Upsilon(\sL)$ can be written uniquely as $\nu= \mu+\gamma$ where $\mu \in \Pi(\sL)$ and  $\gamma \in \Sigma(\sL)$. Any $\nu \in L^*_\infty(X,\sL,\l)$ can be written uniquely as
\begin{equation}\label{deceqn}\nu = \mu +\gamma \in ( L^*_\infty(X,\sL,\l)\cap\Pi(\sL)) \oplus ( L^*_\infty(X,\sL,\l)\cap \Sigma(\sL)).\end{equation}
 If $\nu \geq 0$ the elements of the decomposition are non-negative. This is the Yosida-Hewitt Decomposition of finitely additive measures.
\end{theorem}

By \eqref{deceqn},  $\nu = \mu +\gamma$ where  $\mu \in ( L^*_\infty(X,\sL,\l)\cap\Pi(\sL))$ and $\l \gg\gamma\in \Sigma(\sL)$. If $(X,\sL,\l)$ is $\sigma$-finite,  by the Lebesgue-Radon-Nikodym Theorem
\cite[Ch. 3.8]{folland}   there exists $g \in L_1(X,\sL,\l)$
with   \begin{equation}\label{lrnthm}\int_X u \,d\gamma = \int_X u g\,d\l \forall  u \in L_\infty(X, \sL,\l).\end{equation}
In this case \eqref{deceqn} can be re-written
\begin{equation}
\label{deca}\nu = \mu + g\l,\quad \mu \in \Pi(\sL)\cap  L^*_\infty(X,\sL,\l),\quad g\in L_1(X,\sL,\l).
\end{equation}
The relation between this and  the Lebesgue decomposition of Borel measures is the topic of Section \ref{restriction}.

\subsection{$\boldsymbol{\mathfrak{G}}:$  0-1 Measures}\label{mfg}
Recall that $L^*_\infty(X,\sL,\l)$ is the set of  finitely additive measures on $\sL$ that are zero on $\sN$.  Let
\begin{equation}\label{mathfrakg}\mathfrak{G}= \{\omega \in L^*_\infty(X,\sL,\l):~\omega (X) = 1\text{ and } \omega(A) \in \{0,1\} \forall A \in \sL\}.\end{equation}
$A \in \sL$ is called a $\l$-atom if $\l(A) >0$ and if $A \supset E\in \sL$ implies $\l(E) \in \{0,\l(A)\}$.

\begin{theorem}\label{gammameasures} Suppose  $\om \in \mathfrak{G}$.
(a) Either $\om \in \Pi(\sL)$ or $\om \in \Sigma(\sL)$.
(b) If $(X,\sL,\l)$ is $\sigma$-finite and $~\om \in \Sigma(\sL)$, there exists a $\l$-atom $E_\om$ such that $\om(E) = \l(E\cap E_\om)/\l(E_\om)$   for all $E \in \sL$.
\end{theorem}
\begin{remark*} Hence $\mG\subset \Pi(\sL)$   if $(X,\sL,\l)$ is $\sigma$-finite and  $\sL$ has no $\l$-atoms.
A stronger statement,  Lemma \ref{gammeas}, can be made when $\sL$ is the Borel $\sigma$-algebra of a locally compact Hausdorff space.\qed
\end{remark*}
\begin{proof} (a) For $\om \in \mathfrak{G}$, by Theorem \ref{decomposition},  $\om = \mu + \gamma$ where $\mu \in \Pi(\sL)$ and $\gamma\ll \l$, $\gamma \in \Sigma(\sL)$ are non-negative.
By Theorem \ref{1.22} there exists $\{E_k\} \subset \sL$ with  $\mu(E_k) = \mu(X)$ for all $k$ and $\gamma (E_k) \to 0$ as $k \to \infty$.
If $\om(E_k) = 0$ for some $k$ then $0 = \om(E_k) \geq \mu(E_k) =\mu(X)$ and so $\om = \gamma \in \Sigma(\sL)$. If $\om (E_k) = 1$ for all $k$, then
$$1 = \om(E_k) = \mu(E_k) + \gamma (E_k) = \mu(X) + \gamma(E_k) \to \mu(X) \text{ as } k \to \infty.
$$
Hence $\om(X) = 1 = \mu(X)$ and consequently $\gamma(X) = 0$. Thus $\om = \mu \in \Pi(\sL)$.

(b) Since  $\om \ll \l$ where $\om\in \Sigma(\sL)$ is finite and $\l$ is $\sigma$-additive,  by \eqref{lrnthm} there exists $g \in L_1(X,\sL,\l)$ with $\om(E) = \int_E g\,d\l$ for all $E \in \sL$.
So $g \geq 0$ $\l$-almost everywhere on $X$.
Since $g \in L_1(X,\sL,\l)$,  $\l \left(\{x\in X: g(x) \geq n\}\right) \to 0$ as $n \to \infty,$  and hence, by \cite[Cor. 3.6]{folland},
$$\om \left(\{x\in X: g(x) \geq n\}\right) = \int_{\{x\in X: g(x) \geq n\}} g\,d\l \to 0 \text{ as $n \to \infty$} .
$$
Since $\om \in \mG$ it follows that $\om \left(\{x\in X: g(x) \geq N\}\right) =0$  for some $N\in \NN$.  Now, by finite additivity, $\om(X) = 1$ and $\om (E) \in \{0,1\}$ implies that for every $K \in \NN$ there exists a unique
$k_{_K} \in \{1,\cdots, N2^K\}$ such that
$$1=\om(X) = \sum_{k=1}^{N2^K} \om (E_k)= \om (E_{k_{_K}}) \text{ where $E_k = \left\{x\in X:\frac{k-1}{2^K} \leq g(x) <\frac{k}{2^K} \right\}$}. $$ Hence $E_{k_{_{K+1}}} \subset E_{k_{_K}} $ and since $\om$ is $\sigma$-additive it follows that $\om(E_\om) =1$ where
 $E_\om = \{x \in X: g(x) = \a\}$ for some $\a \in [0,N]$. Then $\l(E_\om)>0$ because $\om(E_\om) = 1$ and, $ \forall E \in \sL$,
$$\om (E) = \om (E \cap E_\om) = \int_{E\cap E_\om} \a\,d\l = \a\l(E\cap E_\om).
$$
Hence $\a = 1/\l(E_\om)$, and $E_\om$ is a $\l$-atom with the required properties because $\om \in \mG$ .
\end{proof}

\begin{subequations}\label{uu}
\begin{theorem}\label{ans}
For $u \in \Li$ and $\om \in \mG$ there is a unique $\a \in I := [-\|u\|_\infty, \|u\|_\infty]$ such that
\begin{gather}\om\left(\{x \in X: |u(x) - \a| < \e\}\right) =1 \forall \e >0,\label{u1}
\\
\int_X u \,d \om = \a \text{ and } \int_X |u| \,d \om = |\a|\label{u2}.
\end{gather}
\end{theorem}
\end{subequations}
\begin{remark*}  Thus, on $\Li$  elements of $\mathfrak{G}$ are analogous to Dirac measures $\mathfrak D$ in the theory of continuous functions on topological spaces. When  \eqref{u1} holds we say  that $u =\a$ on $X$ $\om$-almost everywhere even though it does not imply that $\om\left(\{x \in X: u(x) =\a \}\right) =\om(X)$ if $\om \not\in \Sigma(\sL)$.   \qed
\end{remark*}
\begin{proof}Since $\om$ is zero on $\sN$, it is clear that  $\a \in I$ if \eqref{u1} holds.
Now \eqref{u1} cannot hold for distinct $\a_1 <\a_2$ because, with  $\e =  (\a_2-\a_1)/4$ the sets
$\om(\{x \in X: |u(x) - \a_i| < \e\}),~i= 1,\,2$, are disjoint and by finite additivity the $\om$-measure of their union would be 2. Since  $\om\in \mG$, there is at most one  $\a$ for which \eqref{u1} holds.

Now suppose that there is no $\a$ for which \eqref{u1} holds. Then for each $\a \in I$ there is an $\e_\a>0$ such that
$\om\left(\{x \in X: |u(x) - \a| < \e_\a\}\right)=0$. By compactness there exists $\{\a_1,\cdots, \a_K\} \subset I$ such that  $I \subset \cup_{k=1}^K (\a_k -\e_{\a_k}, \a_k +\e_{\a_k})$ and consequently
\begin{align*}1 =\om (X)  &=\om \left(\{x: u(x) \in\cup_{k=1}^K (\a_k -\e_{\a_k}, \a_k +\e_{\a_k})\right)\\& \leq \sum_{k=1}^K \om\left(\{x: u(x) \in(\a_k -\e_{\a_k}, \a_k +\e_{\a_k})\}\right) = 0.
\end{align*}
Hence  \eqref{u1} holds for a unique $\a$. The first part of \eqref{u2} follows because, by \eqref{u1}, $u =\a $ $\om$-almost everywhere on $ X$ and $\om(X) = 1$. Finally,  $\om\left(\{x \in X: \big| |u(x)| - |\a| \big| < \e\}\right) =1 \forall \e >0$, and the second part of \eqref{u2} follows.
\end{proof}
 The next result  give
the existence elements of $\mG$.
\begin{theorem}\label{thm1*} \textnormal{\cite[Thm. 4.1]{yosidaetal}} Let $\sE\subset \sL\setminus \sN$ have the property that $E_\ell \in \sE,\,1 \leq \ell\leq L$
implies that $\cap_{\ell=1}^L E_\ell \notin \sN$. Then there exists  $\omega \in \mG$ with  $\omega (E) = 1$ for all $E \in \sE$.
\end{theorem}
The proof is by Zorn's lemma and for  given $\sE$ there can be uncountably many $\om$.
The same argument underlies the  correspondence between elements of $\mathfrak{G}$ and ultrafilters.
\begin{definition}\label{ultrafilter} Given $(X,\sL,\l)$, a filter is a family $\sF$ of subsets of $X$  satisfying:
(i) $X \in \sF \text{ and }\sN \cap \sF = \emptyset;$
(ii) $E_1,\,E_2 \in \sF, \text{ implies that } E_1\cap E_2 \in \sF;$
(iii) $E_2\supset E_1\in \sF  \text{ implies that  } E_2 \in \sF.$
A maximal filter $\sF$, one which  satisfies
(iv) $\sF \subset  \hat \sF \text{  implies  }   \sF = \hat \sF$,
is called an ultrafilter.
Let $\Fk$ denote the family of  ultrafilters.
\end{definition}
It is obvious  that  when $\om \in \mathfrak{G}$
\begin{subequations}\label{11ultra}
\begin{equation}
\sF(\om): = \{E \in \sL :~ \om (E) = 1\} \in \mathfrak F.
\end{equation}
Conversely, when $\sF \in \Fk$,
\begin{equation}\om(E): = \left\{\begin{array}{l} 1 \text{ if }E \in \sF\\ 0 \text{ otherwise}\end{array}\right\} \in \mG.\label {2.9b}
\end{equation}
\end{subequations}
This holds because, exactly as in the proof of \cite[Thm. 4.1]{yosidaetal}, the maximality of  $\sF\in \mathfrak F$
implies that for $E \in \sL$ precisely one of $E$ and $X\setminus E$ is in $\sF$.
Thus \eqref{2.9b} defines $\om \in \mG$  with $\sF = \sF(\om)$ and hence $\om \leftrightarrow \sF(\om)$ is a one-to-one correspondence between $\mG$
and  $\Fk$.

By the essential range of $u$ is meant the set
\begin{equation} \label{ERange}
\sR(u):= \left\{\a \in \RR: \l\big(\{x: |u(x)-\a| <\e\}\big)>0 \forall \e>0\right\}.
\end{equation}
\begin{corollary}\label{erange}
For $u \in L_\infty(X,\sL, \l)$,
$$\left\{\int_X u\,d\om:\quad\om \in \mathfrak{G}\right\} = \sR(u).
$$
\end{corollary}

\begin{proof} It follows from Theorems \ref{ans} and \ref{thm1*}  that the right side
is a subset of the left. Since  $\om (E) =1,\,E \in \sL$,  implies $\l(E) >0$, it is immediate from Theorem \ref{ans} that the right side contains the left.
\end{proof}
In a topological space   \eqref{11ultra}, \eqref{ERange} and Corollary \ref{erange} can be localized to points, \eqref{GF},  \eqref{srux} and \eqref{localsR}.

For $A \in \sL$, let
$\Delta_A = \{\omega \in \mathfrak{G}:~\omega (A) = 1\}$ and let $\{ \Delta_A: A \in \sL\}$ be a base for the topology $\mathfrak{t}$ on $\mG$.
  Note from Theorem \ref{thm1*} that $\Delta_A$ is empty if and only if $A\in \sN$ and $\Delta_A$ is both open and closed because $\mathfrak{G} \setminus \Delta_A = \Delta_{X \setminus A}$.
For $u \in L_\infty(X,\sL, \l)$  let $L[u]: \mathfrak{G} \to \RR$ be defined by
\begin{equation} \label{Feqn}
L[u](\omega) = \int_X u \,d\omega \forall \omega \in \mathfrak{G}.
\end{equation}
\begin{theorem}\label{thm10}\textnormal{\cite[Thms. 4.2 \& 4.3]{yosidaetal}} (a) $(\mathfrak{G}, \mathfrak{t})$ is a compact Hausdorff topological space.\\
(b) For $u \in \Li$, $L[u]$ is continuous on $(\mathfrak{G}, \mathfrak{t})$ with
$$ \|u\|_\infty = \|L[u]\|_{C(\mG,\mathfrak{t})}\colon = \sup_{\omega \in \mathfrak{G}} |L[u](\omega)|,$$
and  $u \mapsto L[u]$ is linear from $L_\infty(X,\sL, \l)$ to $C(\mathfrak{G},\mathfrak{t})$.
Moreover, for
$u,v \in L_\infty(X,\sL, \l)$,
\begin{equation}\label{uveqn}L[u](\omega)L[v](\omega) = L[uv](\omega) \forall \omega \in \mathfrak{G}.
\end{equation}
Conversely, for every real-valued continuous function $U$ on
$(\mathfrak{G}, \mathfrak{t})$ there exists $u \in L_\infty(X,\sL,\l)$ with $U = L[u]$. So $L$ is an isometric isomorphism between Banach algebras $L_\infty(X,\sL, \l)$ and $C(\mathfrak{G}, \mathfrak{t})$.
\end{theorem}
 Since $L_\infty(X,\sL,\l)$ and ${C(\mathfrak{G}, \mathfrak{t})}$  are isometrically isomorphic, $u_k\wk u_0$  in $L_\infty (X, \sL, \l)$ if and only if  $L[u_k] \wk L[u_0]$ in $C(\mathfrak{G}, \mathfrak{t})$.
Since $(\mathfrak{G}, \mathfrak{t})$ is a compact Hausdorff topological space, it follows from the opening remarks of the Introduction that $L[u_k] \rightharpoonup L[u_0]$ in $C(\mathfrak{G},\mathfrak{t})$ if and only if  $\{\|L[u_k]\|_{C(\mG,\mathfrak{t})}\}$ is bounded and $L[u_k] \to L[u_0]$ pointwise on $\mathfrak{G}$. Hence  $u_k \rightharpoonup u_0\text{ in }L_\infty(X,\sL, \l)$ if and only if
\begin{equation}\label{key}
 \|u_k\|_\infty \leq M \text{ and } \int_X u_k \,d\om \to \int_X u_0 \,d\om \text{ as } k \to \infty \text{ for all $\om \in \mathfrak{G}$}.
\end{equation}
Sequential weak continuity of composition operators  is an obvious consequence.
\begin{theorem}\label{compositions}
 If $u^n_k \wk  u_0^n$ in $L_\infty(X,\sL,\l)$ as $k \to \infty$,  $n \in \{1,\cdots,N\}$, and  $F:\RR^N \to \RR$ is continuous,  then
$F(u^{1}_k, \cdots , u^{N}_k) \wk F(u_0^{1}, \cdots , u_0^{N})$ in $L_\infty(X,\sL,\l)$.
\end{theorem}
\begin{proof} When $u^n_k \wk u^n_0 $ in $L_\infty(X,\sL,\l)$,  $L[u^n_k] \wk L[u^n_0] $ in $C(\mathfrak{G}, \mathfrak{t})$ and consequently $L[u^n_k](\om) \to L[u^n_0](\om)$ pointwise in $\mathfrak{G}$ as $k \to \infty$.
Therefore, for continuous $F$,  $$F\big(L[u^{1}_k](\om), \cdots , L[u^{N}_k](\om)\big) \to F\big(L[u^{1}_0](\om), \cdots , L[u^{N}_0](\om)\big),~\om \in \mathfrak{G}. $$
If $F$ is a polynomial it follows from \eqref{uveqn} that
$$L[F\big(u^{1}_k, \cdots , u^{N}_k\big)](\om) \to L[F\big(u^{1}_0, \cdots , u^{N}_0\big)](\om),~\om \in \mathfrak{G},
$$
and this holds for continuous $F$, by approximation. Consequently, for continuous $F$,
$$
L[F\big(u^{1}_k, \cdots , u^{N}_k\big)]\wk L[F\big(u^{1}_0, \cdots , u^{N}_0\big)]\text{ in } C(\mathfrak{G}, \mathfrak{t})
$$
and so
$F\big(u^{1}_k, \cdots , u^{N}_k\big) \wk F\big(u^{1}_0, \cdots , u^{N}_0\big)] \text{ in } \Li .
$
\end{proof}
\section{Pointwise and Weak Convergence in $\boldsymbol{\Li}$}\label{wkcgce}
The goal is to characterise weakly null sequences in $\Li$ in terms of their pointwise behaviour, but we begin with some observations on the pointwise behaviour of weakly convergent sequences.
\begin{lemma} \label{iffweak} In $\Li$, $u_k \wk 0$ if and only if $|u_k| \wk 0$.
\end{lemma}
\begin{proof} `Only if' follows from Theorem \ref{compositions} and `if' is a consequence of \eqref{key} since $u_k = u_k^+ - u_k^-$, $0\leq u_k\pm\leq |u_k|$ and $\om \geq 0$.
\end{proof}
\begin{lemma} If  $(X,\sL,\l)$  is $\sigma$-finite and $\{u_k\}$ is weakly null,
 there is  a subsequence $\{u_{k_j}\}$ with $u_{k_j}(x) \to 0$  $\l$-almost everywhere on $X$.
\end{lemma}
\begin{proof} Since $(X,\sL,\l)$  is $\sigma$-finite there exists $f \in L_1(X,\sL,\l)$  which is positive almost everywhere. Since  $|u_k|f \to 0$ in $L_1(X,\sL,\l)$, there is a subsequence with $|u_{k_j}(x)|\to 0$  for $\l$-almost all $x \in X$.
\end{proof}
\begin{lemma}\label{ap1}
Suppose that $(X,\rho)$ is a metric space on which $\l$ is a regular Borel measure with the property that for all locally integrable functions $f$ and balls $B(x,r)$ centred at $x$ and radius $r$,
\begin{equation}\label{hyp}\lim_{0<r \to 0}\fint_{B(x,r)} f d\l = f(x) \text{ for $\l$-almost all }x \in X
\text{ where }
\fint_{B(x,r  )} f d\l: = \frac {1}{\l(B(x,r))}\int f\,d\l.
\end{equation}
Then $u_k \wk u_0$ in $\Li$ implies that $u_k(x) \to u_0(x)$ pointwise $\l$-almost everywhere.
\end{lemma}
\begin{remark}\label{r2}
From \cite[Ch. 1]{heinonen},  \eqref{hyp} holds in particular when $\l$
is a doubling  measure on $(X,\rho)$ (i.e. there exists a
constant $C$ such that $\l(B(x,2r)) \leq C \l (B(x,r))$, or on $\RR^n$ with the standard metric
when $\l$ is any Radon measure (i.e. $\l$ is finite on compact sets).\qed
\end{remark}
\begin{proof}
By hypothesis, for $u \in \Li$ there exists a set $E(u) \in \sB$ with $\l(X\setminus E(u)) = 0$ and
\begin{align}\label{leb1}
 u(x):&=\lim_{0<r \to 0}\fint_{B(x,r)} u\,d\l \text{ for all $x \in E(u)$. }
\end{align}

Now for $u_k \rightharpoonup u_0$ the set $E = \cap_0^\infty E(u_k)$ has full measure. Let $V$ denote the subspace  of $\Li$ spanned by $\{u_k: k \geq 0\}$ and
 for fixed $x \in E$  define a linear functional $\ell_x$ on $V$  by $\ell_x(u) = u(x)$. Then
$$|\ell_x(u)| = |u(x)| = \left|\lim_{0<r \to 0} \fint_{B(x,r)} u\,d \l\right| \leq \|u\|_\infty, ~ u \in V,$$ and, by the
Hahn-Banach Theorem, there exists $L_x \in \Li^*$ with $L_x(u) = \ell_x(u)$ for all $u \in V$. Therefore since $u_k \rightharpoonup u_0$,
$$ u_k(x) = \ell_x(u_k) =   L_x(u_k) \to L_x(u_0) = \ell(u_0) = u_0(x) \forall x \in E.
$$
Hence $u_k \rightharpoonup u_0$ in $\Li$ implies
$u_k(x) \to u_0(x) \text{ for almost all  } x \in X$.
\end{proof}

\begin{remark}\label{app3}
By contrast there follows an example where $\{\|u_k\|_\infty\}$ is bounded,   $u_k$ is continuous except at one point and $u_k(x) \to 0$ everywhere as $k \to \infty$, but $u_k \not\wk 0$ in $\Li$.
Let $X = (-1,1)$, for each $k>2$ let $u_k(0) = 0, ~ u_k(x)=0$ when $|x|\geq 2/k$, $u_k(x) = 1$ if $0<|x|\leq 1/k$, and linear elsewhere. Now in Theorem \ref{thm1*} let $E_\ell = (-1/2\ell,0)\cup (0, 1/2\ell)$ for each $\ell$ and  let $\om$ be a finitely additive measure that takes the value 1 on $E_\ell$ for all $\ell$. Then $\om \in \mG$ and, by Theorem \ref{ans}, $\int_X u_k \,d\om = 1 \forall k$. So $u_k \not\wk 0$, yet
 it is clear that $u_k (x) \to 0$ for all $x \in X$. \qed \end{remark}

By a well-known result of Mazur,  $y_k \wk y$ in a normed linear space  implies, for any  strictly increasing sequence $\{k_j\}$ in $\NN$, that  some $\{\ol y_i\}$ in the convex hull of $\{y_{k_j}:j \in \NN\}$ converges strongly  to $y$. Hence if $u_k \wk 0$ in $\Li$, by
Lemma \ref{iffweak} there exists $\{\ol u_i\}$  in the convex hull of $\{|u_{k_j}|:j\in\NN\}$ with
$$ \ol u_i \to 0 \text{ as }i \to \infty\text{ and, } \forall i,~ \ol u_i = \sum_{j=1}^{m_i} \gamma^i_j |u_{k_j}|, \quad \gamma^i_j \in [0,1] \text{ and }  \sum_{j=1}^{m_i} \gamma^i_j = 1 , \text{ for some } m_i \in \NN.
$$
Since $\gamma^i_j$ may be zero there is no loss in assuming that $\{m_i\}$ is increasing. Therefore, for a strictly increasing sequence $\{k_j\}$ in $\NN$,
$$ 0 \leq w_i(x): = \inf\left\{|u_{k_j}(x)|: j \in \{1,\cdots, m_i\}\right\} \leq \ol u_i(x), ~ x \in X,
$$
defines a non-increasing sequence  in $\Li$  with $\|w_i\|_\infty \to 0$. It follows that if $u_k \wk 0$
\begin{equation}\label{vj} v_J(x) =  \inf\left\{|u_{k_j}(x)|: j \in \{1,\cdots, J\}\right\}
\end{equation}
is a non-increasing sequence  in $\Li$  with $\|v_J\|_\infty \to 0$ as $J \to \infty$.
We now show that a sequence is weakly null in $\Li$ if and only if every sequence $\{v_J\}$, defined as above in terms of a strictly increasing $\{k_j\}$,  converges strongly to 0 in $\Li$. To do so, for
 $u \in \Li$ and $\a >0$, let
$$A_\a(u)= \{x\in X: |u(x)| >\a\}.
$$
\begin{theorem} \label{thmIFF} A bounded sequence $\{u_k\}$ in $\Li$ converges weakly to zero if and only if for every $\a>0$ and every strictly increasing sequence $\{k_j\}$ in $\NN$ there exists $J\in \NN$ with the property that
\begin{equation}\label{IFF}
\l \left\{ \cap_{j=1}^J A_\a(u_{k_j})\right\} = 0.
\end{equation}
This criterion  is equivalent to saying that for every strictly increasing sequence $\{k_j\}$ in $\NN$ the corresponding sequence $\{v_J\}$ in \eqref{vj}
converges strongly to zero in $\Li$.
\end{theorem}
\begin{proof}
Suppose, for a strictly increasing sequence $\{k_j\}$ and $\a >0$ , that \eqref{IFF} is false for all $J \in \NN$. Then $\sE = \{A_\a(u_{k_j}): j \in \NN\}$ satisfies the hypothesis of Theorem \ref{thm1*}. Hence there exists $\om \in \mG$ such that $\om(A_\a(u_{k_j})) = 1$ for all $j$. It follows that
$$\int_X|u_{k_j}|d\om \geq \int_{A_\a(u_{k_j})}|u_{k_j}|d\om \geq \a >0 \forall j.
$$
Hence $|u_k| \not\wk 0$ by \eqref{key} and so, by Lemma \ref{iffweak}, $u_k \not\wk 0$.

Conversely suppose $u_k \not \wk 0$. Then by Lemma \ref{iffweak} and \eqref{key}, there exists $\a>0$, a strictly increasing sequence $\{k_j\}\subset \NN$ and $\om \in \mG$ such that
$$\int_X|u_{k_j}| d\om =:\a_j >\a >0\forall j \in \NN.
$$
Since $\a_j -\a >0$, by Theorem \ref{ans},
$$\om\big(\{x:\big||u_{k_j}| - \a_j\big|<\a_j -\a\}\big) =1 \forall j.
$$
Therefore, since $|u_{k_j}| -\a = |u_{k_j}| -\a_j + \a_j -\a$, it follows that
$\om(A_\a(u_{k_j})) = 1 \forall j$. Hence, since $\om$ is a 0-1 measure, by finite additivity
$\om\Big(\cap_{j=1}^J A_\a(u_{k_j})\Big) = 1$ for all $J$. Since $\om \in \mG\subset \Lid$, it follows that \eqref{IFF} is false for all $J$.
Finally note that for a strictly increasing sequence $\{k_j\}$ and $\a>0$,
$$\l\{x: v_J(x)>\a\} = \l\big\{x: |u_{k_j}(x)| >\a \forall j \in \{ 1,\cdots, J\}\big\} = \l\big\{\cap_{j=1}^JA_\a(u_{k_j})\big\}.
$$
Since $v_J(x) \geq v_{J+1}(x) \geq 0$ it follows that $v_J \to 0$ in $\Li$ if and only if \eqref{IFF} holds for every $\a>0$. This completes the proof.
\end{proof}
There follows an analogue of Dini's theorem that on  compact topological spaces monotone,
pointwise convergence of sequences of continuous functions to a continuous function is uniform; equivalently, for bounded monotone sequences  weak and strong convergence coincide.
\begin{corollary} \label{cor3.6} Suppose $\{u_k\}$ is bounded in $\Li$ and $|u_k(x)| \geq |u_{k+1}(x)|$, $k \in \NN$,  for $\l$-almost all $x \in X$. Then $u_k \wk 0$ if and only if $u_k \to 0$ in $\Li$.
\end{corollary}
\begin{proof} The monotonicity of $\{|u_k|\}$ implies that $v_J$ coincides with $|u_J|$ in Theorem \ref{thmIFF} and so that $|u_J| \to 0$ in $\Li$ as $J \to \infty$ when $u_k \wk 0$ if  in $\Li$.
The converse is obvious.
\end{proof}

\subsection{Illustrations of Theorem \ref{thmIFF}}\label{S3.1}

(1) In this example $X= (-1,1)$ with  Lebesgue measure,  $u_k$ is  supported in $[-1/2, 1/2]$, $\|u_k\|_\infty = 1$ and $u_k,\,u_k^- \wk 0$,  but $u_k^+ \not\wk 0$ where $u_k^\pm (x) =
u_k(x  \pm 1/2^{k+1})$.
To see this, let
$A_k = [1/2^{k+1}, 1/2^k),~ A_k^\pm = A_k \mp 1/2^{k+1}$, $u_k = \chi_{A_k}$ and  $ u_k^\pm = \chi_{A_k^\pm}$.
Clearly $u_k^\pm(x) = u_k(x \pm 1/2^{k+1})$ and $u _k^+ \not \wk 0$ because  $v_J$, defined in \eqref{vj} by $u_k^+$, is 1  on $ (0, 1/2^{J+1})$. But
since $\{A_k\}$ and $\{A_k^-\}$ are two mutually disjoint families,  in \eqref{vj} $v_J$, defined for any  $\{k_j\}\subset \NN$ by $u_{k_j}$ or $u_{k_j}^-$   , is zero for $J \geq 2$. Hence $u_k  \wk 0$ and $u_k^- \wk 0$.
That $\chi_{A_k}\wk 0$ for a disjoint family  of sets
is used in Remark \ref{necessuf}.\qed

(2)  In $\Li$ let
$
u_k(x) = \sum_{i=1}^\infty \a_i \chi_{A_k^i}, x \in X,
$
where $\Sigma_{i=1}^\infty|\a_i| < \infty$ and, for each $i \in \NN$, $\{A^i_k\}_{k\in \NN}$ is a family of mutually disjoint non-null measurable sets. Then $u_k \wk 0$ in  $\Li$.

To see this, note that for each $x \in X$ and  $i\in \NN$ there exists at most one $k \in \NN$, denoted, if it exists, by $\kappa(x,i)$, such that $x \in A^i_k$ if and only if $k = \kappa(x,i)$. Note also that for $\e>0$ there exists
$I_\e\in \NN$ such that $\Sigma_{I_\e+1}^\infty |\a_i|< \e$. Hence, for  any given $k \in \NN$ and $x \in X$,
$$
|u_k(x)| \leq \sum_{i=1}^{I_\e} |\a_i| \chi_{A_k^i}(x) + \e=
\begin{array}{l}
\sum\limits_{\substack{i \in \{ 1, \cdots, I_\e\}\\\kappa (x,i)=k,\\}}|\a_i| +\e.  \end{array}
$$
Since  $\{\kappa(x,i): i \in \{1,\cdots , I_\e\}\}$ has at most $I_\e$ elements, there exists $k \in \{1,\cdots, I_\e+1\}$ such that $k \neq \kappa(x,i)$ for any $i \in \{1,\cdots, I_\e\}$. Consequently
$\inf\{|u_k(x)|: 1 \leq k \leq I_\e+1\} \leq \e$, independent of $x \in X$.
Since this argument can be repeated with $k \in \NN$ replaced by any strictly increasing subsequence $\{k_j\}$, it follows that $\{v_J\}$ defined in terms of any subsequence in \eqref{vj} has $\|v_J\|_\infty \to 0$ in $\Li$. The weak convergence of $\{u_k\}$ follows. For the special case, take $\a_1 =1$ and $\a_i = 0,\, i \geq 2$.\qed

(3) Let $u:\RR \to \RR$ be essentially bounded and measurable with $|u(x)| \to 0$ as $|x| \to \infty$ and let $u_k(x) = u(x+k)$. Then $u_k \wk 0$ in $\Li$ where $\l$ is Lebesgue measure on $\RR$. To see this, for $\e>0$
suppose that $|u(x)| <\e$ if $|x| > K_\e$. The for any  $\{k_j\} \subset \NN$, $\|v_J\|_\infty < \e $ for all $J \geq K_\e$ where $\{v_J\}$ is defined in terms of $\{u_{k_j}\}$ by \eqref{vj}, and the result follows. \qed

(4) Let $u:\RR \to \RR$ be essentially bounded and measurable with $u(x) \to 0$ as $x \to \infty$ and $u(x) \to 1$ as $x \to -\infty$. Let $u_k(x) = u(x+k)$. Then $u_k(x) \to 0$ as $k \to \infty$ for all $x \in \RR$,
but $u_k$ is not weakly convergent to 0 because of Theorem \ref{thmIFF}. However, in the notation of Definition \ref{weakpwcgce}, $u_k \wk 0$ at every point of $\RR$, but not at the point at infinity in its one-point compactification. \qed

(5) Define $\{u_k\}_{k \in \NN}\subset \Li$ by $ u_k(x) = \sin (1/(kx))$, $x \in X = (0,2\pi)$, with the standard measure $\l$ on the Lebesgue $\sigma$-algebra on $X$.
Clearly $|u_k(x)| \to 0$ as $k \to \infty$ uniformly on $(\e, 2\pi)$ for any $\e \in (0, 2\pi)$. Therefore if  a subsequence $\{u_{k_j}\}$ is weakly convergent, its  weak limit must be zero.

To see that no subsequence of $\{u_k\}$ is  weakly convergent to 0, consider first a strictly increasing sequence $\{k_j\}$ of natural numbers for which there exists a prime power $p^m$ which does not divide $k_j$ for all $j$. Then, for $J \in \NN$ sufficiently large, let
$$
x_J =\left\{\frac{\pi}{p^m} \text{lcm}\,\{k_1,\cdots ,k_J\}\right\}^{-1} \in (0,2\pi),
$$
where $\text{lcm}$ denotes the  least common multiple. Then, since $p^m \nmid  k_j$ and $p$ is prime,
$$
\frac{1}{k_jx_J} = \frac{\text{lcm}\,\{k_1,\cdots ,k_J\}}{p^m k_j}\,\pi \quad\text{ where }\quad \frac{\text{lcm}\,\{k_1,\cdots ,k_J\}}{ k_j}  = r\text{\,mod\,}p^m, \quad r \in \{1,\cdots, p^m-1\},
$$
from which it follows that $|u_{k_j}(x_J)| \geq |\sin (\pi/p^m)| >0$, independent of $J$.
Since, for all $j \in \{1,\cdots, J\}$, $u_{k_j}$ is continuous at $x_J$, it follows that $\|v_J\|_{\Li} \geq |\sin (\pi/p^m)| >0$ for all $J$ sufficiently large.
By Theorem \ref{thmIFF} this shows that $u_{k_j} \not\wk 0$ if $\{k_j\}$ has a subsequence $\{k'_j\}$for which $p^m \nmid k'_j$ for all $j \in \NN$.
Note that if this hypothesis is not satisfied by $\{k_j\}$ for any prime $p$ and $m\in \NN$, then every $K\in \NN$ is a divisor of $k_j$ for all $j$ sufficiently large, how large depending on $K$.
Consequently, if
$u_{k_j} \wk 0$, $\{k_j\}$ has subsequence $\{k'_{j}\}$  with the
property that  $2^{j+2}k'_{j}$ divides $k'_{j+1}$ for all $j$.
In other words $2^{j+2} n_{j}k'_{j} = k'_{j+1}$, $n_{j} \in \NN$, and  $[0, k'_{j+1})$ is a union of $2^{j+2} n_j$ disjoint intervals of length $k'_{j}$.

Now fixed $J \in \NN$, let $m_J$ denote the mid-point of $I_J := [0, k'_J)$,
and let $I_{J-1}:= [m_J-k'_{J-1}, m_J)$, which is a half open interval of length $k'_{J-1}$ to the left of $m_J$.
Then
$$x = r\mod k'_{J} \text{ where } r\in \left[\frac {k'_J}2 \left(1- \frac{1}{2^{J}n_{J-1}}\right), \frac {k'_J}2\right] \forall x \in I_{J-1},$$ since $2^{J+1} n_{J-1}k'_{J-1} = k'_{J}$.
Note that $m_J$, and hence the  end-points of $I_{J-1}$, are   integer multiples of $k'_{J-1}$. Now denote the mid-point of $I_{J-1}$  by $m_{J-1}$  and let $I_{J-2} = [m_{J-1} -k'_{J-2}, m_{J-1})$, the interval
of length $k'_{J-2}$ to the left of $m_{J-1}$. Then $I_{J-2} \subset I_{J-1} \subset I_J$ and
 $$x = r\mod k'_{J-1} \text{ where } r\in \left[\frac {k'_{J-1}}2 \left(1- \frac{1}{2^{J-1}n_{J-2}}\right), \frac {k'_{J-1}}2\right] \forall x \in I_{J-2},$$ since $2^{J} n_{J-2}k'_{J-2} = k'_{J-1}$.
Repeating this construction leads to a nested sequence of intervals, $I_{0} \subset I_1 \subset \cdots  \subset I_{J-1}$ with the property that
 $$x = r\mod k'_{i+1} \text{ where } r\in \left[\frac {k'_{i+1}}2 \left(1- \frac{1}{2^{i+1}n_{i}}\right), \frac {k'_{i+1}}2\right] \forall x \in I_{i}.$$
Now let $x_J = \{m_0 \pi\}^{-1}$ where $m_0 \in I_0$. Since $I_0= \cap_{i=0}^{J-1}I_i$ and  $\half\left(1- \frac{1}{2^{i+1}n_{i}}\right) \geq \frac 14$ for $i \geq 0$,
$$|u_{k'_j}(x_J)| = \sin \left(\frac{m_0}{k'_j}\pi\right) \geq \sin (\pi/4) = \frac{1}{\sqrt 2}, ~j \in \{1,\cdots, J\},$$
 independent of $J \in \NN$. Hence  $\{v_J\}$ defined by \eqref{vj} using the subsequence $\{k'_j\}$ does not converge to 0 in $\Li$. It follows that $u_{k_j} \not\wk 0$ as $j \to \infty$ for any $\{k_j\} \subset \NN$.\qed
\section{ $\boldsymbol{L_\infty^*(X, \sB,\l)}$ when $\boldsymbol{(X, \tau)}$ is a Topological Space}\label{topsection}
This section deals with $L_\infty^*(X,\sB,\l)$ when  $(X,\tau)$  is a locally compact Hausdorff topological space, $\sB$ is the corresponding Borel $\sigma$-algebra and  $\l\geq 0$ is a measure on $\sB$ as described in Section \ref{Linfty*}. In addition here  $\l$ is assumed regular and finite on compact sets.
  In that setting  a regular Borel measure that takes only values 0 or 1 is a Dirac measure concentrated at a point $x_0 \in X$.  As before, $\mG$ is defined by \eqref{mathfrakg}.

\begin{lemma}\label{gammeas} For $\om \in \mG$ there exists a compact set $K \in \sB$ with $\om(K)=1$ if and only if there exists $x_0 \in X$ such that $\om(G) = 1$ for all open sets $G$ with $x_0 \in G$.  For all $\om \in \mG$ there is at most one such $x_0$ and when $(X,\tau)$ is compact there is exactly one such $x_0$.
\end{lemma}
\begin{proof}
Suppose that $\om(K)=1$, $K$ compact, and  the result is false. Then for  $x \in K$ there is an open $G_x$ with $x \in G_x$ and $\om(G_x)=0$. By compactness,  $K \subset \cup_{i=1}^K G_{x_i}$ where $\om (G_{x_i}) = 0,\,1\leq i\leq K$, which implies $\om(K) = 0$. Since this is false, $\om(K)=1$ for compact $K$  implies the existence of $x_0 \in K$ with the required property. Since $X$ is Hausdorff, if there is another  $x_1 \in X$ with this property
 there are open sets with $x_0 \in G_{x_0},\, x_1 \in G_{x_1}$ and  $G_{x_0}\cap G_{x_1} = \emptyset$. But this is impossible because by finite additivity $\om \big(G_{x_0}\cup G_{x_1}\big) = 2$.
Now suppose that $\om(K) = 0$ for all compact sets $K$. By local compactness, for  $x \in X$ there is an open set $G_{x}$ with $x \in G_{x}$ and  its closure $\ol{G_{x}}$ is compact. Since   $\om (G_x) \leq\om(\ol{G_x}) = 0$, there is no $x\in X$ with the required property. Finally, the existence of $x_0$ when $X$ is compact follows because $\om(X) = 1$. This completes the proof. \end{proof}

   Let $(X_\infty, \tau_\infty)$ denote the one-point compactification \cite{kelley} of $(X,\tau)$. Then $X_\infty = X \cup \{x_\infty\}$,  $x_\infty \notin X$ (the ``point at infinity''),  and a subset $G$ of $X_\infty$  is open if either
 $G \subset X$ is open in $X$, or $G = \{x_\infty\} \cup(X \setminus K)$ for some compact  $K \subset X$. Then $(X_\infty,\tau_\infty)$ is a compact Hausdorff topological space  because $(X,\tau)$ is locally compact  Hausdorff, and  $(X,\tau)$ is compact if and only if $\{x_\infty\}$ is an isolated point (open and closed) in $(X_\infty,\tau_\infty)$.
  For $\om \in \mathfrak G$, let $\om_\infty$ be defined on Borel subsets $E$ of $X_\infty$ by $\om_\infty(E) = \om (E\cap X)$. Then $\om_\infty$ is the unique finitely additive measure on $X_\infty$ which takes only values 0 and 1 and  coincides with  $\om$ on Borel sets in $X$.  In this setting Lemma \ref{gammeas} can be
 re-stated:
\begin{lemma}\label{newcor}   Let  $(X,\tau)$ be  a locally compact Hausdorff space and $\om \in \mG$. Then
there exists a unique $x_0 \in X_\infty$ such that $\om_\infty(G) = 1$ for all open sets $G$ in $X_\infty$ with $x_0 \in G$;  $x_0 = x_\infty$ if and only if $\om (K) = 0$ for all compact $K \subset X$ and  $x_0 \in X$ if $(X,\tau)$ is compact.
\end{lemma}
\subsection{Localization of Weak Convergence in $\boldsymbol{L_\infty(X,\sB,\l)}$}\label{localisation}
By \eqref{11ultra} there is a one-to-one correspondence between $\mG$ and $\Fk$. For  $x_0 \in X_\infty$, let $\mG(x_0) \subset \mG$ denote the set of $\om \in \mG$  for which the conclusions of Lemma \ref{gammeas} holds, and let  $\Fk(x_0)\subset \Fk$  be the corresponding family of ultrafilters. Then, by Lemma \ref{gammeas},
\begin{equation}\label{GF}
\mG= \cup_{x_0 \in X_\infty} \mG(x_0), \qquad \Fk= \cup_{x_0 \in X_\infty} \Fk(x_0),\end{equation}
which leads to the following definition of weak pointwise convergence.
\begin{definition}\label{weakpwcgce}
  $u_k$ converges weakly to $u$ at $x_0 \in X_\infty$ if
$$\int_X u_k\,d\om \to \int_X u \,d\om \forall \om \in \mG(x_0).
$$
\end{definition}
The localized version of Theorem \ref{thmIFF} is immediate.
For $u \in \Li$, $\a >0$ and $E \in \sL$ let $A_\a(u|_E)= \{x\in E: |u(x)| >\a\}.$
\begin{theorem} \label{thmIFFlocal} A bounded sequence $\{u_k\}$ in $L_\infty (X,\sB,\l)$ converges weakly to zero at $x_0 \in X_\infty$  if and only if  for every $\a>0$, every strictly increasing sequence $\{k_j\}$ in $\NN$ and every open $G \subset X_\infty$ with $x_0 \in G$ there exists $J$ with
$\l \big\{ \cap_{j=1}^J A_\a(u_{k_j}|_G)\big\} = 0$. Equivalently, in \eqref{vj}, $v_J \to 0$ in $L_\infty(G, \sB,\l)$.
\end{theorem}

By analogy with  \eqref{ERange}, for $x_0 \in X_\infty$ the essential range of $u$ at $x_0\in X_\infty$ is defined by
\begin{equation}\label{srux}\sR(u)(x_0) = \left\{\int_X u\, d\om: \om \in \mG(x_0)\right\}.
\end{equation}
As in Corollary \ref{erange}, for open  $G$ with $x_0 \in G$,
\begin{align}\label{localsR}
\sR(u)(x_0)
= \left\{\int_G u\, d\om: \om \in \mG(x_0)\right\} = \left\{\a: \l\{x\in G: |\a - u(x)| < \e\} > 0\forall \e >0\right\}.
\end{align}
Note that $\sR(u)(x_0)$ is closed in $\RR$  because, by \eqref{localsR}, for any  $x_0 \in X$  its complement is open.
It is immediate  from \eqref{key}, Lemmas \ref{gammeas} and \ref{newcor} that $u_k \wk u$ in $L_\infty(X,\sB,\l)$ if and only if for all $x_0 \in X_\infty$
$$\a_k:=\int_X u_k\,d\om \to \int_X u\,d\om =:\a\text{ as } k \to \infty \forall \om \in \mG(x_0),
$$
which is not equivalent to $\a_k \to \a$ when $\a_k \in\sR (u_k)(x_0) $ and  $\a \in \sR (u)(x_0)$ because, possibly,
$$\a_k = \int_X u_k\,d\om_k \text{ and } \a= \int_X u \,d\om, \text{ but }\om_k \neq \om.
$$
However, $\a = \int_X u\,d\om \in\sR(u)(x_0),  ~\om \in \mG(x_0)$, may be thought of as
   a directional limit of $u$ at $x_0$,  the ``direction'' being determined by $\sF(\om)\in \Fk(x_0)$. Then weak convergence in $L_\infty(X,\sB, \l)$ is equivalent to convergence, for each $\sF \in \Fk(x_0)$, of the directional limits  of $u_k$ at $x_0$ to corresponding directional limits of $u$ at $x_0$,  for each $x_0 \in X_\infty$.
Therefore, by Theorem \ref{ans}, $u_k \wk u$ in $L_\infty(X,\sB,\l)$ if and only if for all $x_0 \in X_\infty$  and all $\om \in \mG (x_0)$
\begin{subequations} \label{mult}
\begin{multline}|\a_k -\a|\to 0 \text{ and }  \om\left\{x \in G: |u_k(x) -\a_k| + |u(x)-\a| <\e\right\} =1\\  \forall \e>0
\text{ and all open sets $G\subset X_\infty$ with $x_0 \in G$},
\end{multline}
equivalently
 $u_k \wk u$ in $L_\infty(X,\sB,\l)$ if and only if for all $x_0 \in X_\infty$  and all $\sF \in \Fk (x_0)$,
\begin{multline}|\a_k -\a|\to 0 \text{ and }  \left\{x \in G: |u_k(x) -\a_k| + |u(x)-\a| <\e\right\} \in \sF\\  \forall \e>0
\text{ and all open sets $G\subset X_\infty$ with $x_0 \in G$}.
\end{multline}
\end{subequations}

\begin{remark}\label{necessuf}
It follows that for  $u_k \wk u$ it is \emph{necessary} that for every $x_0 \in X_\infty$ and every $\alpha \in \mathcal{R}(u)(x_0)$ there exist
$\alpha_k \in \mathcal{R}(u_k)(x_0)$ such that $\alpha_k \to \alpha$ as $k \to \infty$ and \emph{sufficient} that for every $x_0 \in X_\infty$
$$
\sup\left\{|\gamma| : \ \gamma \in \mathcal{R}(u_k - u)(x_0)\right\} \to 0 \ \mbox{ as } \ k \to \infty .
$$
As noted earlier,
the necessary condition is not sufficient. To see that the sufficient condition is not necessary,
let $u_k = \chi_{A_k}$ where  $\{A_k\}$    is a sequence of disjoint segments
 centred on the origin 0 of the unit disc $X$ in $\RR^2$.
 Then $\sR(u_k)(0) = \{0,1\}$ but $u_k \wk 0$ by  the last remark in Section \ref{S3.1} (1) or, equivalently, by Section \ref{S3.1} (2)  with $\a_1 = 1$ and $\a_i = 0,~ i \geq 2$. In this example   $\int_X u_k d\om \to 0$, but not uniformly, for every $\om \in \mG(0)$.
\end{remark}

\section{Restriction to $\boldsymbol{C_0(X, \tau)}$ of Elements of $\boldsymbol{L_\infty^*(X, \sB,\l)}$}\label{restriction}
Throughout this section $(X,\tau)$ is a locally compact Hausdorff  topological space and $C_0(X,\tau)$ is the space of real-valued continuous functions $v$ on $X$  with the property that  for all $\e>0$ there exists a compact set $K \subset X$ such that $|v(x)| <\e$ for all $x \in X \setminus K$.
 When endowed with the maximum norm
\begin{equation}\label{ctau}\|v\|_\infty = \max_{x \in X}|v(x)|,\quad v \in C_0(X,\tau),
\end{equation} $C_0(X,\tau)$ is a Banach  space which if $X$ is compact consists of all real-valued continuous functions on $X$.   Let $\nu \in L_\infty^*(X,\sB,\l)$, as in Theorem \ref{yhthm}
define $f \in L_\infty(X, \sB,\l)^*$ by
$$f(u) = \int_X u \,d\nu, ~ u \in L_\infty(X, \sB,\l),
$$
 and let $\hat f$ denote the restriction of $f$ to  $C_0(X,\tau)$. By  the
 Riesz Representation theorem \cite[Thm. 6.19]{rudin} there is a unique  bounded regular Borel measure $ \hat \nu \in \Sigma (\sB)$  corresponding to $\hat f$, and consequently
  \begin{equation}\label{nudefn} \int_X v\, d\nu =  \int_X v\, d\hat\nu \forall v \in C_0(X,\tau).
  \end{equation}
The goal  is to understand how $\hat \nu$ depends on $\nu$ and, since ${\hat \nu}^\pm = \widehat{\nu^\pm}$ (see \eqref{pm}), there is no loss of generality in restricting attention to  non-negative $\nu \in L_\infty^*(X,\sB,\l)$.
 Recall
 \begin{itemize}
 \item[(i)] from the Yosida-Hewitt decomposition \eqref{deca},  $\nu  = \mu +g\l$  where   $\mu \in L_\infty^*(X,\sB,\l)$  is purely finitely additive and
 $g\l$, $g \in L_1(X,\sB,\l)$, is $\sigma$-additive.
  \item[(ii)] from the Lebesgue-Radon-Nikodym Theorem \cite[Thm. 3.8]{folland}, \cite[Thm. 6.10]{rudin}, $\hat \nu = \rho +k\l$ where $\rho$ and $k\l$ are $\sigma$-additive, $k \in L_1(X,\sB,\l)$, and  $\rho$ is singular with respect to $\l$. Thus $\hat \nu$  has a singularity with respect to $\l$ if $\hat \nu (E)\neq 0$ (equivalently $\rho(E) \neq 0$) for some $E \in \sN$, and  $\hat \nu$   is singular if in addition $k = 0$, where \\
\begin{equation} \label {decb}
\int_X v\,d \mu +\int_X vg\,d\l = \int_X v\,d \nu = \int_X v\,d\hat \nu =   \int_X v\,d\rho + \int_X v k\,d\l \forall v \in C_0(X,\tau),
\end{equation}
 \end{itemize}
 where  $\rho \perp \l$ in $\Sigma(\sB)$, $ \mu \perp \l$ in $\Upsilon (\sB)$  (see Remark \ref{absol} for the distinction), and $g,k\in L_1(X,\sB,\l)$.
Valadier was first to note that  the relation between $ \mu$ and $\rho$, and $g$ and $k$  is not straightforward.
\begin{theorem*}[ Valadier  \cite{valadier}]\label{valadier}
When $\l$ is  Lebesgue measure on $[0,1]$ there is a non-negative $\nu \in \Pi (\sB)$ with
$$\int_0^1 v\,d\nu=\int_0^1 v\,d\l  \forall v \in C[0,1].
$$
Thus  in (i), (ii) and \eqref{decb},  $0\neq \mu \in \Pi(\sB)$ and $g=0$ but  $\rho =0$ and $k \equiv 1$, and $\hat \nu$ has no singularity.
\end{theorem*}
Hensgen independently   observed that the last claim in \cite[Theorem 3.4]{yosidaetal} is false.
\begin{theorem*}[Hensgen \cite{hensgen}]
With $X = (0,1)$ there exists $\nu\in L_\infty^*(X, \sB,\l)$ which is non-zero and not purely finitely additive but
 $\int_0^1 v\, d\nu =0$ for all $v \in C(0,1)$.
\end{theorem*}
Subsequently Abramovich \& Wickstead \cite{abramwick}
provided wide ranging generalizations and recently  Wrobel \cite{wrobel} gave a sufficient condition on   $\nu$  for $\hat \nu$  to be singular with respect to Lebesgue measure on $[0,1]$.
To find a formula for $\hat\nu$ satisfying \eqref{nudefn} for a given non-negative  $\nu \in L_\infty^*(X, \sB,\l)$, and to characterise those $\nu$ for which $\hat \nu$ has a singularity, recall the following version of Urysohn's Lemma.
\begin{lemma} \textnormal{\cite[\S 2.12]{rudin}}
  \label{urysohnlcs}  If   $K \subset G \subset X$ where $K$ is compact and $G$ is open,  there exists a continuous function $f : X \to [0, 1]$ such that $f(K)=1$, $\ol{\{x:f(x) >0\}}\subset G$ is compact, and hence $f \in C_0(X,\tau)$.
\end{lemma}

\begin{lemma}\label{howd} Suppose $0\leq \nu \in L_\infty^*(X,\sB,\l)$ and $B \in \sB$. Then $\nu(K) \leq \hat \nu (B) \leq \nu(G)$  for  compact $K$ and open $G$ with $K \subset B \subset G$. Moreover
\begin{equation*}
\nu(K) \leq \hat \nu(K) \text{ and } \hat \nu (G) \leq \nu (G) \text{ for all compact  $K$ and open $G$ in $X$ }
\end{equation*}
and
$\nu (F) \leq \hat \nu (F) +\nu(X) -\hat\nu(X)$  if $F$ is closed.
Thus  $\hat \nu (X) = \nu (X)$ implies that  $\nu(F) \leq \hat \nu(F)$ for all closed sets $F \subset X$. (That $\nu(X) = \hat\nu(X)$ when $(X,\tau)$ is compact  was noted following \eqref{nudefn}.)
\end{lemma}

\begin{proof} For a given Borel set $B$ and $K \subset B \subset G$  as in the statement,
let $f$ be the continuous function determined  in Lemma  \ref{urysohnlcs} by $K$ and $G$. Then
\begin{gather*}
\nu(K) \leq \int_X f\,d\nu \leq \nu (G) \text{ and } \hat\nu(K) \leq \int_X f\,d\hat\nu \leq \hat\nu (G).
\end{gather*}
It follows from \eqref{nudefn} that
$\nu(K) \leq \hat \nu (G)$  and  $\hat \nu (K) \leq \nu (G)$ whence, since $\hat \nu$ is regular \cite[Thm. 6.19]{rudin},  $\nu(K) \leq \hat\nu(B) \leq \nu (G)$. In particular, if $B = K$ is compact, $\nu(K) \leq \hat\nu(K)$, and if $B = G$ is open,
$\hat\nu(G) \leq \nu(G)$. That $\nu (F) \leq \hat \nu (F) +\nu(X) -\hat\nu(X)$  when $F$ is closed  follows by finite additivity since $0\leq \nu(X),\, \hat \nu(X)<\infty$.
\end{proof}

\begin{remark}\label{alexandroff}
 A non-negative finitely additive set function $\nu$ on $\sB$  is said to be regular \cite[III.5.11]{dunford} if for all $E \in \sB$ and $\e>0$ there are sets $F \subset E \subset G$ with $F$ closed, $G$ open and $\nu(G \setminus F) < \e$.
If $X$ is compact and $\nu$ is regular, by a theorem of Alexandroff \cite[III.5.13]{dunford} $\nu$ is $\sigma$-additive and hence $\hat\nu = \nu$. By Lemma \ref{howd},  if $\nu(X) = \hat \nu(X)$  and $F\subset E\subset G$, where $F$ is closed and $G$ is open,
$$\nu(F) \leq \hat \nu (F) \leq \hat \nu(E) \leq \hat \nu(G) \leq \nu(G).
$$
Hence  if $\nu(X) = \hat \nu(X)$ and $\nu\geq 0$ regular implies that  $\nu = \hat \nu$  is $\sigma$-additive on $\sB$.\qed
\end{remark}
\begin{theorem}\label{newsing}  Suppose $K$ is compact, $G$ is open, $K \subset G$ and $0\leq \nu \in L_\infty^*(X,\sB,\l)$. Then for $n \in \NN$ there exist compact  $K_n$ and  open $G_n$ with
\begin{gather*}K  \subset G_n\subset K_n\subset G,\quad G_n \subset G_{n-1},~K_n \subset K_{n-1},\\ \hat\nu (K) \leq \nu(K_n),\quad \hat\nu (G) \geq \nu(G_n)
\text{ and }\l (K_n) < \l(K) + 1/n.\end{gather*}\end{theorem}
\begin{proof}
Since $\l$ is a regular Borel measure that is finite on compact sets there exist open sets $G^k$ with $K \subset G^k \subset G $ and $\l(G^k) < \l(K) +1/k$ for $k \in \NN$.
By Lemma \ref{urysohnlcs} there exists a continuous function $f_k:X \to [0,1]$ such that $f_k(K) = 1$ and $\ol{\{x: f_k(x)>0\}}$ is a compact subset of $G^k$. For $x \in X$, let $g_n(x) = \min \{ f_k(x): k\leq n\}$ so that $g_n \leq g_{n-1}$, $g_n$ is continuous on $X$,
$g_n(K) = 1$ and $\ol{\{x: g_n(x)>0\}}\subset G^n$ is compact.

Let $G_n= \{x: g_n(x)>0\}$ and $K_n = \ol{\{x: g_n(x)>0\}}$. Then
$K\subset G_n\subset K_n\subset G^n \subset G$
and, by Lemma \ref{howd},
 $$\hat \nu(K) \leq \hat\nu (G_n) \leq  \nu (G_n) \leq  \nu (K_n), \quad \hat\nu(G) \geq \hat\nu(K_n) \geq \nu(K_n) \geq \nu(G_n) ,
  $$
and  $\l (K_n) <\l(K) +1/n$ because $K_n \subset G^n$. Now  $\{G_n\}$ and  $\{K_n\}$ are nested sequences of open and compact sets, respectively, because $g_n(x)$ is decreasing in $n$, with the required properties. \end{proof}
\begin{corollary}\label{hatnuX}
For $G$ open, $K$ compact and $\nu \in L_\infty^*(X,\sB,\l)$ non-negative,
$$\hat \nu(G) = \sup \{\nu( K):   K \subset G, \text{ $K$ compact}\},\quad \hat \nu(K) = \inf \{\nu( G):   K \subset G, \text{ $G$ open}\}.
$$
\end{corollary}
\begin{proof} Let   $G$ be open. Then for any $\e>0$  there exists compact $K_\e\subset G$ with $\hat\nu(K_\e) >\hat\nu(G) -\e$,  since $\hat\nu$ is regular, and   $\nu(K_\e) \leq  \hat \nu (K_\e) \leq \hat\nu(G)$ by Lemma \ref{howd}. Now by Theorem \ref{newsing} there exists compact $K_{1}$ with $K_\e \subset K_1\subset G$ and $\hat\nu(G) \geq  \hat \nu(K_1) \geq \nu(K_1) \geq \hat\nu(K_\e) >\hat\nu(G) -\e$. This establishes the first identity.
Similarly for compact $K$ and  $\e>0$  there exists  open $G_\e$ with $K \subset G_\e$ and $\hat\nu(G_\e) <\hat\nu(K) +\e$, and an open $G_1$ with  $\hat \nu (G_\e) \geq \nu(G_1)$, $K \subset G_1\subset G_\e$, whence $\hat \nu (K) +\e >\hat\nu(G_\e) \geq \nu(G_1)$.
\end{proof}

\begin{corollary}\label{singular1}  For  $0\leq \nu \in L_\infty^*(X,\sB,\l)$, $\hat \nu \in \Sigma (\sB)$ has a singularity if and only if there exists $\a >0$ and a sequence    of compact sets with $\nu(K_n) \geq \a $,  $K_{n+1}\subset K_n$ for  all $n$, and $\l(K_n) \to 0$ as $n \to \infty$.
\end{corollary}
\begin{proof} If $\a >0$ and such a sequence exists, by Lemma \ref{howd}, $\hat \nu (K_n) \geq \a$  for all $n$. Since $\{K_n\}$ is nested and $\hat \nu$ is $\sigma$-additive it follows that $\hat \nu (K) \geq \a$ where $K = \cap_n K_n$. Since $K\in \sN$, because $\lim_{n\to \infty} \l(K_n) = 0$ and $\l$ is
$\sigma$-additive, $\nu$ has a singularity.
Conversely if $\hat\nu\geq 0$ has a singularity there exists $E \in \sN$ and $\a >0$ with $\hat \nu (E) = 2\a$. Since $\hat \nu$ is regular, there exists a compact $K \subset E$ \ with $\hat \nu (K) \geq\a >0$. Now since $\l(K) = 0$ because $K \subset E \in \sN$, the existence of  compact sets with $\nu(K_n) \geq \hat\nu (K) \geq \a $,  $K_{n+1}\subset K_n$ for  all $n$, and $\l(K_n) \to 0$ as $n \to \infty$ follows from Theorem \ref{newsing}.
  \end{proof}
  \begin{theorem}\label{measures}
For  $B \in \sB$ and $0\leq \nu \in L_\infty^*(X,\sB,\l)$,
\begin{equation} \label{bbb}\hat \nu(B) = \inf_{_{\substack{\,G\, \textnormal{open}\\B \subset G}}}\left\{\sup_{_{\substack{\,K\,\textnormal{compact}\\K\subset G}}}\nu(K) \right\}
= \sup_{_{\substack{\,K\, \textnormal{compact}\\K \subset B}}}\left\{\inf_{_{\substack{\,G\,\textnormal{open}\\K\subset G}}}\nu(G) \right\}.
\end{equation}\end{theorem}
\begin{proof}
This follows from Corollary \ref{hatnuX} since $\hat \nu$ is a regular Borel measure.
\end{proof}
\begin{corollary}\label{surprise} (a) For  $\om \in \mG$,
 either $\hat\om$ is zero or $\hat \om$ is a Dirac measure.
(b) Both possibilities in (a) may occur when $(X,\tau)$ is not compact.
(c) If   $\hat\om = \d_{x_0} \in \mathfrak D$, then $\om \in \mG(x_0)$.
\end{corollary}
\begin{proof}
(a) If $\om (K) = 0$ for all compact $K$, the first formula of \eqref{bbb} implies that  $\hat\om = 0$. If $\om(K) = 1$ for some compact $K$, by Lemma \ref{gammeas}  there is a unique $x_0 \in X$ for which $\om (G) =1$ if $x_0 \in G$ and $G$ is open. From the second part of \eqref{bbb} it is immediate that $\hat\om(B) =1$ if and only if $x_0 \in B$. Hence $\hat\om \in \mathfrak D$.

(b) For an example of both possibilities let $X = (0,1)$ with the standard locally compact topology  and Lebesgue measure. Let $\om\in \mG$ be defined by  Theorem \ref{thm1*} with $E_\ell = (0, 1/\ell),~\ell \in \NN$. Then $\om (K) = 0$ for all compact $K\subset (0,1)$ and hence $\hat\om = 0$. On the other hand if $E_\ell = (1/2+1/\ell, 1/2)$ in Theorem \ref{thm1*}, $\om\in \mG$ with $\om ([1/2+1/\ell, 1/2]) =1$ for all $\ell$ and hence $\hat\om= \d_{1/2} \in \mathfrak D$.

(c) When   $\hat\om = \d_{x_0}$ let $x_0 \in G $, open. Since $\{x_0\}$ is compact by Lemma \ref{urysohnlcs} there exists $v \in C_0(X,\tau)$ with $v(X) \subset [0,1]$,
$v(x_0) = 1$, $v(X \setminus G) = 0$. Now  $\om(G) = 1$ for every open set with $x_0\in G$ since
$$1 \geq \om (G) \geq \int_G v\,d\om =\int_X v\,d\om = \int_X v\,d\hat\om = v(x_0)=1 .
$$
\end{proof}

\appendix
\section{Appendix: {\bf $\boldsymbol \mG$ and Extreme Points of the Unit Ball in $\boldsymbol \Lid$ }} \label{apA}
\begin{theorem}[Rainwater \cite{rainwater}] \label{rainthm} In a Banach space $B$, $x_k \wk x$ as $k \to \infty$ if and only if $f(x_k) \to f(x)$ in $\RR$ for all  extreme points $f$ of the closed  unit ball in $B^*$.\qed\end{theorem}

Recall that $L^*_\infty(X,\sL,\l)$ is the set of  finitely additive measures on $\sL$ that are zero on $\sN$.
Let $U^*_\infty = \{\nu \in \Lid:|\nu|(X) \leq 1\}$, the closed unit ball in $\Lid$. Then  $\nu \in U^*_\infty$  is an extreme point of $U^*_\infty$ if for $\nu_1,\nu_2 \in U^*_\infty$ and $\a \in (0,1)$
$$\nu(E) = \a \nu_1(E) + (1-\a)\nu_2(E) \forall E \in \sL \text{ implies that } \nu =\nu_1 =\nu_2.
$$
Clearly  extreme points have $|\nu|(X) =1$ and, by Theorem \ref{rainthm},  $u_k \rightharpoonup u_0\text{ in }L_\infty(X,\sL, \l)$ if and only if for some $M$
\begin{equation*}
 \|u_k\|_\infty \leq M \text{ and } \int_X u_k \,d\nu \to \int_X u_0 \,d\nu \text{ as } k \to \infty \text{ for all extreme points $\nu$  of $U^*_\infty$}.
\end{equation*}
Thus \eqref{key} is a consequence of the following result.
\begin{lemma}\label{kkey} $\nu$ is an extreme point of $U^*_\infty$ if and only if either $\nu$ or $-\nu \in \mG$, see \eqref{mathfrakg}.
\end{lemma}
\begin{proof}
If $|\nu|(X) = 1$ but $\nu$ is not one signed, then $|\nu| = \nu^+ + \nu^-$ where $\nu^+\wedge \nu^- =0$ and $\nu^+(X)\in (0,1)$. Let $0<\e_0 = \frac 12 \min\{ \nu^+(X), 1-\nu^+(X)\}$ and, by \eqref{pm},  choose $A \in \sL$ such that $\nu^+(X\setminus A) + \nu^-(A) = \e < \e_0$. If $\nu(A) = 0$ then $\nu^+(X) = \nu^+(X\setminus A) + \nu^+( A) = \nu^+(X\setminus A) + \nu^-( A) = \e < \e_0$, which is false. So $\nu(A) \neq 0$ and hence $|\nu|(A) >0$. If $|\nu|(A) = 1$ then $\nu^+(X)= 1 +\e -2\nu^-(A) \geq 1 -\e$, and hence $1-\nu^+(X) <\e <\e_0$, which is false. So $|\nu|(A) \in (0,1)$. Let
$$\nu_1(E) = \frac{\nu(A \cap E)}{|\nu|(A)},\quad \nu_2(E) = \frac{\nu((X \setminus A) \cap E)}{|\nu|(X\setminus A)} \forall E \in \sL.
$$
Then $\nu_1,\,\nu_2 \in U^*_\infty$ and, for all $E \in \sL$,
$$  \nu(E) = \a \nu_1(E) + (1-\a) \nu_2(E),\text{ where } \a = |\nu|(A),~(1-\a) = |\nu|(X\setminus A).
$$
Since $\a \in (0,1)$, $\nu_1(A) =\nu(A) /|\nu|(A)\neq 0$ and $\nu_2(A) = 0$, this shows that $\nu$ is not an extreme element of $U^*_\infty$ if $\nu$ is not one-signed.

Now suppose $0\leq \nu \in U^*_\infty$ (for $\nu \leq 0$ replace $\nu$ with $-\nu$).  If $\nu \notin \mG$ there exists $A \in \sL$ with $\nu(A) \in (0,1)$. Let
$$\nu_1(E) = \frac{\nu(A \cap E)}{\nu(A)},\quad \nu_2(E) = \frac{\nu((X \setminus A) \cap E)}{\nu(X\setminus A)} \forall E \in \sL.
$$
Then $\nu_1,\,\nu_2 \in U^*_\infty$,
$$  \nu(E) = \a \nu_1(E) + (1-\a) \nu_2(E) \forall E \in \sL,\text{ where } \a = \nu(A),~(1-\a) = \nu(X\setminus A).
$$
Since $\nu_1(A) = 1\neq 0 =\nu_2(A)$, $\nu$ is not extreme. Hence $\nu$  extreme implies that $\pm\nu \in\mG$.

Now suppose that $\nu\in \mG$ and  for all $E \in \sL$,
\begin{equation*}\nu(E) = \a \nu_1(E) + (1-\a) \nu_2(E),\quad \a\in (0,1),\quad \nu_1,\nu_2 \in U^*_\infty.
\end{equation*}
Then $\nu \geq 0$ and if $\nu(E) = 1$,
$$1=\nu(E) = \a \nu_1(E) + (1-\a) \nu_2(E) \leq  \a |\nu_1|(X) + (1-\a) |\nu_2|(X) \leq 1$$
which implies that
$\nu_1(E) =\nu_2(E) = \nu (E)= 1$. In particular $\nu_1(X) =\nu_2(X) = 1$. If $\nu(E) = 0$ then
$\nu(X\setminus E) = 1$ and so    $\nu_1(X\setminus E) =\nu_2(X \setminus E) = 1$, whence $\nu_1(E) =\nu_2(E)= \nu(E)= 0$. Thus $\nu = \nu_1 = \nu_2$ and
$\nu$ is  extreme if $\nu \in \mG$.
\end{proof}

{\bf Closing Remark.}
Although the main result, Theorem \ref{thmIFF}, is derived above from Yosida-Hewitt theory \cite{yosidaetal} without reference to other sources, Theorem \ref{rainthm} and
   Lemma \ref{kkey} lead to \eqref{key}, and Lemma \ref{ans} yields Corollary \ref{iffweak}, thus dispensing with any need for Theorems \ref{thm10} and \ref{compositions}. \qed

\textbf{Acknowledgement.}  I am grateful to Charles Stuart for guiding me to the original source \cite[Annexe]{banach}, to Anthony Wickstead for copies of \cite{valadier} and \cite{wrobel}, to Geoffrey Burton, Brian Davies and Eugene Shargorodsky for their interest and  comments, and  to Mauricio Lobos whose question led to this study. \qed

\date{\today}
\ed

 A particular consequence is the following.
\begin{corollary}\label{charfns} ~If $\{A_k\}\subset \sL$ is a sequence of non-null, disjoint sets, then $\|\chi_{A_k}\|_\infty = 1$ and $\chi_{A_k} \wk 0$ as $k \to \infty$ in $L_\infty(X, \sL, \l)$, where $\chi_{A_k}$ denotes the
characteristic function of $A_k$.
\end{corollary}
\begin{proof}
By Theorem \ref{yhthm} for any $f \in \Li^*$ there is a finitely additive measure $\nu$ with
$$f(\chi_{A_k}) = \int_X \chi_{A_k}\,d\nu = \nu(A_k) = \nu^+(A_k) - \nu^-(A_k),$$
where $\nu^\pm$ are non-negative finitely additive measures (see \eqref{pmparts} below).
Since, by finite additivity,
$$0 \leq \sum_{k=1}^K \nu^\pm (A_k) = \nu^\pm \left(\cup_{k=1}^K A_k\right)\leq \nu^\pm(X) \leq |\nu|(X) = \|f\|_\infty< \infty,~ \forall K \in \NN,
$$
and the result follows.
\end{proof}

\subsection*{Ultrafilters in $\boldsymbol \Li$}

It is helpful to identify $\om \in \mathfrak{G}$ with families of subsets of $X$, just as Dirac measures $\mathfrak D$ are identified with singletons.
We then show that if $u \in L_\infty(X,\sB,\l)$ and $\om \in \mathfrak{G}$, then $u$ is constant $\om$-almost everywhere on $X$, just as any continuous function is constant  $\d$-almost everywhere where $\d\in \mathfrak D$.
 Thus $\mathfrak{G}$ is the analogue  for essentially bounded functions of $\mathfrak D$ for continuous functions.
For $\om \in \mathfrak{G}$, let
$\sF(\om) = \{E \in \sB :~ \om (E) = 1\}$.
Then
\begin{align*}
(i)\quad &
\sN \cap \sF(\om) = \emptyset  ;
\\
(ii) \quad & \text{ if }  E_1,\,E_2 \in \sF(\om), \text{ then } E_1\cap E_2 \in \sF(\om)    ;
\\
(iii) \quad & \text{ if }  E_2\supset E_1\in \sF(\om),  \text{ then } E_2 \in \sF(\om)    ;
\\
(iv) \quad & \text{ if }    \sF(\om) \subset  \sF(\hat\om), \text{  then }   \sF(\om) = \hat \sF(\om)  .
\end{align*}
In Boolean algebra any family of subsets of $L_\infty^*(X,\sB,\l)$ satisfying (i)-(iii) is called a filter, and a filter which is maximal with respect to set inclusion  is
called an ultrafilter. Since $\sF(\om)$ satisfies (i)-(iv) it is an ultrafilter. The proof of Theorem \ref{thm1*} leads to the conclusion that if $\sF$ is an ultrafilter (the existence of which followed from Zorn's lemma) and $E \in \sB$, then exactly one of $E$ and $X\setminus E$ is in $\sF$,  after which the existence of the finitely additive measure with $\sF = \sF(\om)$  is straightforward.

Let $\Fk$ denote the family of ultrafilters in $L_\infty^*(X,\sB,\l)$. Since, for  $\sF \in \Fk$ there is a unique $\om \in \mathfrak{G}$
with $\sF = \sF(\om)$ the relation $\sF(\om) = \sF$ defines a one-to-one correspondence between $\Fk$
and  $\mathfrak{G}$.

In a metrizable space something sharper can be said.
\begin{theorem}\label{measures}
Let $(X,\tau)$ be  locally compact and metrizable and let $0\leq \nu \in L_\infty^*(X,\sB,\l)$.

 (a) When $K$ is compact and $G$  is open,
$$\nu(K) \leq \hat \nu(K) \text{ and } \hat \nu (G) \leq \nu (G).
$$
(b) For a compact set $K$,
$$\hat \nu(K) = \inf\{ \nu(G): K \subset G,\,G \textnormal{ open}\} = \lim_{k \to \infty}\nu(F_k)  = \lim_{k \to \infty}\nu(G_k),
$$
where
$$ F_k = \{ x: \textnormal{dist}\,\{x,K\} \leq 1/k\} \text{ and } G_k = \{ x: \textnormal{dist}\,\{x,K\} < 1/k\},~k\in\NN.
$$

(c) For an open set $G$,
$$\hat \nu(G) = \sup\{ \nu(F): F \subset G,\,F \textnormal{ closed}\}  = \lim_{k \to \infty}\nu(G^k)  = \lim_{k \to \infty}\nu(F^k).
$$
where, with $^cG = X\setminus G$,
$$ G^k = \{ x: \textnormal{dist}\,\{x,^cG\} > 1/k\} \text{ and } F^k = \{ x: \textnormal{dist}\,\{x,^cG\} \geq 1/k\},~k\in\NN.
$$
(d) For any  Borel set $B$,
$$\hat \nu(B) = \sup_{_{\substack{\,K\, \textnormal{compact}\\K \subset B}}}\left\{\inf_{_{\substack{\,G\,\textnormal{open}\\K\subset G}}}~\nu(G) \right\}= \inf_{_{\substack{\,G\, \textnormal{open}\\B \subset G}}}\left\{\sup_{_{\substack{\,F\,\textnormal{closed}\\F\subset G}}}\nu(F) \right\}
 .
 $$
\end{theorem}
\begin{proof}  (a) This follows from Lemma \ref{howd} since a metric space is Hausdorff.

(b) By part (a),  $K \subset G$ and $G$ open implies $\hat \nu (K) \leq \hat \nu(G) \leq \nu(G)$ and hence
$$\hat \nu(K) \leq \inf\{ \nu(G): K \subset G,\,G \textnormal{ open}\}.
$$
Since $X$ is locally compact, for each $x \in K$ there exists an open set $G_x$ with $x \in G_x$ and  $\ol{G_x}$  compact. Therefore if $K$ is compact, $K \subset G: = \cup_{j=1}^J G_{x_j}$ for some $J \in \NN$. Thus $G$ is open   and any closed subset of  $G$ is compact.
Now  $F_k$ (defined in the statement of (b)) is closed, $F_k \subset G$ for all $k$ sufficiently large and $ \cap_k F_k = F$ because $F$ is closed. Moreover $G_k$ is open and $ G_{k} \subset F_k\subset F_{k-1}\subset G$ for all $k$ sufficiently large. Therefore, by part (a),
$$\hat\nu (K) \leq \hat \nu (G_k) \leq\nu(G_k)\leq \nu(F_{k}) \leq \hat\nu(F_{k})\to \hat \nu (K) \text{ as } k \to \infty,
$$
since $\hat \nu$ is $\sigma$-additive. Hence, when $K$ is compact
$\inf\{ \nu(G): K \subset G,\,G \textnormal{ open}\} \leq \hat \nu (K),
$
which proves  (b). Part (c)  follows by finite additivity since the complement of an open set is closed and vice versa and (d) is immediate from (b) and (c) since $\hat \nu$ is a regular Borel measure.
\end{proof}

It is immediate that the non-negative Borel measure $\hat \nu$ is absolutely continuous with respect to $\l$ (which defines $\sN$) if and only if  $$\inf \{ \nu (G): G ~\text{open},~F \subset  G\} = 0
\text{ for every closed set $F \in \sN$. }$$
 Also $\hat \nu$ is a Dirac measure supported at a singleton $p \in X$, if and only if
 $$\inf \{ \nu (G): G ~\text{open},~p\in G\}>0 \text { and } \nu (F) =0 \text{ when $F$ is closed and } p \notin F.$$

For general non-negative $\nu \in  L^*_\infty(X,\sB,\l)$, the regular measure  $\hat\nu$ is singular with respect to $\l$  if and only if $\hat \nu(K) >0$ for some compact $K \in \sN$. Even though     $\nu (K) = 0$, because $\nu \in L_\infty^*(X, \sB,\l)$, the singular behaviour  of $\hat\nu$ may be characterised in terms of $\nu$ as follows.
\begin{theorem}\label{singular} $0\leq \hat \nu \in \Sigma (\sB)$ has a singularity if and only if there exists $\a >0$ and a sequence    of closed sets with $\nu(F_k) \geq \a $,  $F_{k+1}\subset F_k$ for  all $k$, and $\l(F_k) \to 0$ as $k \to \infty$.
\end{theorem}
\begin{proof} If $\a >0$ and such a sequence exists, by Theorem \ref{measures}(a), $\hat \nu (F_k) \geq \a$  for all $k$. Since $\{F_k\}$ is nested and $\hat \nu$ is $\sigma$-additive it follows that $\hat \nu (F) \geq \a$ where $F = \cap_k F_k$. Since $F\in \sN$, because $\lim_{k\to \infty} \l(F_k) = 0$ and $\l$ is
$\sigma$-additive, $\nu$ has a singularity.

Conversely if $\hat\nu\geq 0$ has a singularity there exists a closed $F \in \sN$ with $\hat \nu (F) = \a >0$. As before, let
$$ F_k = \{ x: \textnormal{dist}\,\{x,F\} \leq 1/k\} \text{ and } G_k = \{ x: \textnormal{dist}\,\{x,F\} < 1/k\}.
$$
Since $F$ is closed, $F_k$ is closed and $ \cap_k F_k = F$. Moreover $G_k$ is open and $F  \subset G_{k} \subset F_{k-1}$, for all $k$. Therefore, by Theorem \ref{measures}(a),
$$0<\a = \hat \nu (F) \leq \hat\nu( G_{k})\leq \nu(G_k)\leq \nu(F_{k-1}) \forall k.
$$
Since $\l(F_k) \to  \l(\cap_k F_k) = \l(F) =0$, this completes the proof.
\end{proof}
The following is immediate by finite additivity.
\begin{corollary}\label{singularopen} $0\leq \hat \nu \in \Sigma (\sB)$ has a singularity if and only if there exists $\a >0$ and a nested sequence, $G_{k+1}\supset G_k$,    of open sets with $\nu(G_k) \leq \nu(X) -\a $ for  all $k$ and $\l(G_k) \to 1$ as $k \to \infty$.
\end{corollary}
\ed
I will use notation from \cite[Section 4]{T18}.

Ideally, I would like to have a better understanding of weak convergence in $L_\infty(X, \mathcal{L}, \lambda)$ in terms of
point-wise essential ranges of the functions involved.

If I am not mistaken, it follows from \cite[Section 4.1]{T18} that for weak convergence of $u_k$ to $u$ in
$L_\infty(X, \mathcal{L}, \lambda)$, it is\\
{\bf (A)}\ necessary that for every $x_0 \in X_\infty$ and every $\alpha \in \mathcal{R}(u)(x_0)$ there exist
$\alpha_k \in \mathcal{R}(u_k)(x_0)$ such that $\alpha_k \to \alpha$ as $k \to \infty$; \\
{\bf (B)}\ sufficient that for every $x_0 \in X_\infty$
\begin{equation}\label{1}
\sup\left\{|\gamma| : \ \gamma \in \mathcal{R}(u_k - u)(x_0)\right\} \to 0 \ \mbox{ as } \ k \to \infty .
\end{equation}

Trivial examples show that {\bf (A)} is not sufficient for weak convergence of $u_k$ to $u$ in
$L_\infty(X, \mathcal{L}, \lambda)$. It would be nice to have an explicit example where  $u_k$ converges to $u$ weakly in
$L_\infty(X, \mathcal{L}, \lambda)$, but {\bf (B)} \underline{\it is not satisfied}.\\

{\bf Example.} \ Let $X = \mathbb{R}$ be equipped with the Lebesgue measure and the standard topology.
Let $u(x) := \sin\frac{1}{x}$. Since the norm in $L^1(\mathbb{R})$ is continuous with respect to shifts, it is easy to
show that $u_k := u\left(\cdot + \frac{1}{\pi k}\right)$ converges to $u$ weakly$^*$ in
$L_\infty(\mathbb{R})$  as $k \to \infty$. On the other hand, $\mathcal{R}(u)(0) = [-1, 1]$, while
$\mathcal{R}(u_k)(0) = \{0\}$. Hence, $u_k$ does not converge to $u$ weakly in
$L_\infty(\mathbb{R})$ according to  {\bf (A)}.

Now, let $v_k(x) := u\left(x + \frac{x |x|}{k}\right) = u\left(x\left(1 + \frac{|x|}{k}\right)\right)$. Using the mean value theorem, one can easily show that
$v_k$ converges to $u$ strongly in $L_\infty(\mathbb{R})$  as $k \to \infty$.

\underline{\it Question:} \ Let $w_k(x) := u\left(x + \frac{x}{k}\right) = u\left(x\left(1 + \frac{1}{k}\right)\right)$. Does $w_k$ converge to $u$ weakly in
$L_\infty(\mathbb{R})$? It is not difficult to see that \eqref{1} with $x_0 = 0$ is not satisfied for $w_k$ in place of $u_k$.

======================================

=======================================

============================================================================================

\subsection{The Pointwise Essential Range  of $\boldsymbol u \in \boldsymbol{L_\infty(\Om),~ \Om \subset  {\mathbf R}^n}$}
As in Appendix \ref{apA},  $L_\infty(\Om)$ is an abbreviation for $L_\infty(X, \sB, \l)$ when $X= \Om \subset \RR^n$ is open with respect to the standard metric, $\sB$ denotes the completion of Borel sets in $\Om$, $\l$ is Lebesgue measure and $E(u)$, the Lebesgue set of  $u \in L_\infty(\Om)$, is defined by \eqref{leb1}.  Then it is immediate from \eqref{localsR} and \eqref{leb1} that $u(x) \in \sR(u)(x)$ for all $x \in E(u)$, and $E(u)$ has full measure. If $u$ is the characteristic function of an open ball $B$ in $\Om$, it is obvious that $\sR(u)(x) = \{0,1\}$ at every point of the boundary of $B$.

More remarkably, at every point of $X$ the essential range can be the same non-trivial closed interval.
To see this, let $\mathbb{T}$ and $\mathbb{D}$ denote the unit circle and the unit disk, respectively, and let
$$
B_\rho(\alpha) := \{z \in \mathbb{C} : \ |z - \alpha| < \rho\} , \ \ \ \alpha \in \mathbb{C} , \ \rho > 0.
$$
Let $f \in H^\infty(\mathbb{D})$. We say that $\zeta \in \mathbb{T}$ is a singularity of $f$ if
$f$ cannot be extended analytically to $\zeta$ (see, e.g., \cite[Ch. II, Section 6]{G06}).

For almost every $\zeta \in \mathbb{T}$, the function $f$ has a nontangential limit at $\zeta$,
which we denote by $f(\zeta)$. We will use the following notation (which is different from the notation used in \cite[Ch. II, Section 6]{G06})
\begin{eqnarray*}
\mathcal{R}(f)(\vartheta_0) := \Big\{\alpha \in \mathbb{C}\ : \ \mu\left(\{\vartheta \in \mathbb{R} : \ |\vartheta - \vartheta_0| < \delta , \
\left|f\left(e^{i\vartheta}\right) - \alpha\right| < \varepsilon\}\right) > 0 \\
\mbox{ for all } \varepsilon, \delta > 0\Big\} , \ \ \ \vartheta_0 \in (-\pi, \pi] ,
\end{eqnarray*}
where $\mu$ denotes the standard Lebesgue measure on $\mathbb{R}$.

\begin{theorem}\label{t1}
Let $f \in H^\infty(\mathbb{D})$ be an inner function and let $e^{i\vartheta_0} \in \mathbb{T}$ be its singularity. Then
$\mathcal{R}(f)(\vartheta_0) = \mathbb{T}$.
\end{theorem}
\begin{proof}
Since $|f| = 1$ almost everywhere on $\mathbb{T}$, one has $\mathcal{R}(f)(\vartheta_0) \subseteq \mathbb{T}$.

Suppose $\alpha \in \mathbb{T}$ does not belong to $\mathcal{R}(f)(\vartheta_0)$. Then there exist
$\varepsilon, \delta \in (0, 1)$ such that
\begin{equation}\label{mu0}
\mu\left(\{\vartheta \in \mathbb{R} : \ |\vartheta - \vartheta_0| < \delta , \
\left|f\left(e^{i\vartheta}\right) - \alpha\right| < \varepsilon\}\right) = 0 .
\end{equation}
Let $\alpha_{\pm \varepsilon}$ be the two points where the circle $\{\zeta \in \mathbb{C} : \ |\zeta - \alpha| = \varepsilon\}$
intersects $\mathbb{T}$ and let $\mathbb{T}(\alpha, \varepsilon)$ and $\mathbf{D}(\alpha, \varepsilon)$ be the intersections
of $\mathbb{T}$ and $\mathbb{D}\cup\mathbb{T}$ with the closed half-plane containing $0$ whose boundary is the straight line going
through the points $\alpha_{\pm \varepsilon}$:
\begin{eqnarray*}
&& \mathbb{T}(\alpha, \varepsilon) := \{\zeta \in \mathbb{T} | \ \mbox{Re}\, \left(\alpha \overline{\zeta}\right) \le
\mbox{Re}\, \left(\alpha \overline{\alpha_{\pm \varepsilon}}\right)\} , \\
&& \mathbf{D}(\alpha, \varepsilon) := \{\zeta \in \mathbb{D}\cup\mathbb{T} | \ \mbox{Re}\, \left(\alpha \overline{\zeta}\right) \le
\mbox{Re}\, \left(\alpha \overline{\alpha_{\pm \varepsilon}}\right)\} .
\end{eqnarray*}
Let $\tau > 0$ be the distance from $\alpha$ to $\mathbf{D}(\alpha, \varepsilon)$. (It is not difficult to see that $\tau$ is the distance
from $\alpha$ to the midpoint of the chord $[\alpha_{-\varepsilon}, \alpha_{\varepsilon}]$ and
$\tau = 1 - \cos(2 \arcsin\frac{\varepsilon}{2})$.)

Consider the Poisson kernel:
$$
P_r(\vartheta) := \frac{1 - r^2}{1 - 2r \cos \vartheta + r^2}\ , \ \ \ 0 \le r < 1 , \ \vartheta \in \mathbb{R} .
$$
There exists $\delta_0 > 0$ such that
\begin{equation}\label{est}
\int_{\delta \le |\vartheta - \vartheta_0| \le \pi} P_r(\theta - \vartheta)\, d\vartheta < \tau/2 , \ \ \
\forall z = r e^{i\theta} \in B_{\delta_0}(e^{i\vartheta_0}) .
\end{equation}

It follows from \eqref{mu0} that $f\left(e^{i\vartheta}\right) \in \mathbb{T}(\alpha, \varepsilon)$ for almost all
$\vartheta$ with $|\vartheta - \vartheta_0| < \delta$. Since $P_r \ge 0$,
$$
\int_{|\vartheta - \vartheta_0| < \delta} P_r(\theta - \vartheta)\, d\vartheta <
\int_{ |\vartheta - \vartheta_0| \le \pi} P_r(\theta - \vartheta)\, d\vartheta = 1 ,
$$
$\mathbb{T}(\alpha, \varepsilon) \subset \mathbf{D}(\alpha, \varepsilon)$ and the latter is a closed convex set containing $0$, one has
\begin{equation}\label{close}
\int_{|\vartheta - \vartheta_0| < \delta} P_r(\theta - \vartheta) f(e^{i\vartheta})\, d\vartheta \in \mathbf{D}(\alpha, \varepsilon) .
\end{equation}
On the other hand, it follows from \eqref{est} that
\begin{equation}\label{far}
\left|\int_{\delta \le |\vartheta - \vartheta_0| \le \pi} P_r(\theta - \vartheta) f(e^{i\vartheta})\, d\vartheta\right| \le
\int_{\delta \le |\vartheta - \vartheta_0| \le \pi} P_r(\theta - \vartheta)\, d\vartheta < \tau/2
\end{equation}
for all $ r e^{i\theta} \in B_{\delta_0}(e^{i\vartheta_0})$. Then it follows from the definition of $\tau$ and from \eqref{close}, \eqref{far}
that
\begin{eqnarray*}
f(z) &=& \int_{|\vartheta - \vartheta_0| \le \pi} P_r(\theta - \vartheta) f(e^{i\vartheta})\, d\vartheta \\
&=& \int_{|\vartheta - \vartheta_0| < \delta} P_r(\theta - \vartheta) f(e^{i\vartheta})\, d\vartheta +
\int_{\delta \le |\vartheta - \vartheta_0| \le \pi} P_r(\theta - \vartheta) f(e^{i\vartheta})\, d\vartheta , \\
&& z = r e^{i\theta} \in B_{\delta_0}(e^{i\vartheta_0})
\end{eqnarray*}
cannot belong to $B_{\tau/2}(\alpha)$, which contradicts \cite[Ch. II, Section 6, Theorem 6.6]{G06}. Hence every
$\alpha \in \mathbb{T}$ has to belong to $\mathcal{R}(f)(\vartheta_0)$. So,
$\mathbb{T} \subseteq \mathcal{R}(f)(\vartheta_0) \subseteq \mathbb{T}$, i.e. $\mathcal{R}(f)(\vartheta_0) = \mathbb{T}$.
\end{proof}

If $f$ is an infinite Blaschke product, such that the closure of its zero set contains $\mathbb{T}$,
or the singular inner function defined by a singular measure with closed support equal to $\mathbb{T}$,
then the set of its singularities equals $\mathbb{T}$ (see \cite[Ch. II, Section 6, the paragraph after Theorem 6.6]{G06}).

\begin{corollary}\label{c1}
Let $f \in H^\infty(\mathbb{D})$ be an inner function whose set of singularities equals $\mathbb{T}$.
Then $\mathcal{R}(f)(\vartheta) = \mathbb{T}$ for every $e^{i\vartheta} \in \mathbb{T}$.
\end{corollary}

{\begin{remark}\label{Ru} Corollary \ref{c1} provides an example of a function $w:(-\pi, \pi)\to \CC$ with
$\mathcal{R}(w)(x) = \mathbb{T}$ for every $x \in (-\pi, \pi)$. Let $u(x) := w(x)$,
$x = (x_1, \dots, x_n) \in (-\pi, \pi)^n$. Then
$\mathcal{R}(u)(x) = \mathbb{T}$ for every $x \in (-\pi, \pi)^n$.
\qed
\end{remark}

\begin{remark}
This shows $$L_\infty^1(\Om): = \{u \in  L_\infty(\Om): \sR u(x) = \{u(x)\} \text{ for almost all } x \in \Om\}$$ is  a proper closed linear subspace of $L_\infty(\Om)$ which contains the space of bounded continuous functions on $\Om$. The example in Remark \ref{app3} shows that even in $L_\infty^1(\Om)$ pointwise convergence everywhere is not  sufficient for weak convergence, although it is sufficient for sequences of continuous functions.
\qed  \end{remark}

=========================================================

To see that $\{u_k\}$ is not weakly convergent to 0 let $\{k_j\}$ denote an increasing sequence of natural numbers for which there exists a prime power, $p^m,$ which does not divide $k_j$ for all $j$ ($p=2$ and $m=1$ when all the $k_j$ are all odd, for example). Then, for $J \in \NN$ sufficiently large let
$$
x_J =\left\{\frac{\pi}{p^m} \text{lcm}\,\{k_1,\cdots ,k_J\}\right\}^{-1} \in (0,2\pi),
$$
where $\text{lcm}$ denotes the  least common multiple. Then, since $p^m \nmid  k_j$ and $p$ is prime,
$$
\frac{1}{k_jx_J} = \frac{\text{lcm}\,\{k_1,\cdots ,k_J\}}{p^m k_j}\,\pi \quad\text{ where }\quad \frac{\text{lcm}\,\{k_1,\cdots ,k_J\}}{ k_j}  = r\text{\,mod\,}p^m, \quad r \in \{1,\cdots, p^m-1\},
$$
from which it follows that $|u_{k_j}(x_J)| \geq |\sin (\pi/p^m)| >0$, independent of $J$.
Since, for all $j \in \{1,\cdots, J\}$, $u_{k_j}$ is continuous at $x_J$, it follows that $\|v_J\|_{\Li} \geq |\sin (\pi/p^m)| >0$ for all $J \in \NN$ sufficiently large. Thus, by Theorem \ref{thmIFF}, $u_{k_j} \not\wk 0$ and hence $\{u_k\} $ has no weak limit in $\Li$ in this case.

=========================================================